\documentclass[preprint]{elsarticle}
\usepackage[english]{babel}
\usepackage{amsfonts}
\usepackage{amsthm}
\usepackage{amsmath}
\usepackage{amscd}
\usepackage[active]{srcltx} % SRC Specials: DVI [Inverse] Search
\usepackage{latexsym}
\usepackage{amssymb}
\usepackage[latin1]{inputenc}
\usepackage{xcolor}
\usepackage{enumitem}   

\numberwithin{equation}{section}

\newtheorem{thm}{Theorem}[section]
\newtheorem{prop}[thm]{Proposition}
\newtheorem{cor}[thm]{Corollary}
\newtheorem*{cor*}{Corollary}
\newtheorem{lema}[thm]{Lemma}
\newtheorem*{lema*}{Lemma}
\newtheorem{hyp}[thm]{Hypotheses}

\theoremstyle{definition}

\newtheorem{prob}{Problem}
\newtheorem*{prob*}{Problem}

\newtheorem*{Def}{Definition}

\newtheorem{example}{Example}
\newtheorem{obs}[thm]{Remark}

\newtheorem*{obs*}{Remark}
\newtheorem*{thm*}{Theorem}
\newtheorem*{prop*}{Proposition}

\newcommand{\PI}[2]{\left\langle \,#1 , #2\, \right\rangle}

\newcommand{\K}[2]{\left[ \,#1 , #2\, \right]}

\newcommand{\set}[1]{\left\{ \,#1\, \right\}}

\newcommand{\setB}[1]{\Big\{ \,#1\, \Big\}}

\newcommand{\parentesis}[1]{\left( \,#1\, \right)}

\newcommand{\ra}{\rightarrow}

\newcommand{\x}{\times}

\newcommand{\CC}{\mathbb{C}}
\newcommand{\RR}{\mathbb{R}}
\newcommand{\NN}{\mathbb{N}}
\newcommand{\Z}{\mathbb{Z}}
\newcommand{\St}{\mathcal{S}}

\newcommand{\M}{\mathcal{M}}

\newcommand{\HH}{\mathcal{H}}
\newcommand{\KK}{\mathcal{K}}
\newcommand{\EE}{\mathcal{E}}

\newcommand{\mc}[1]{\mathcal{#1}}

\newcommand{\ol}{\overline}

\newcommand{\ort}{{[\bot]}}

\newcommand{\noi}{\noindent}
\newcommand{\sdo}{[\dotplus]}

\newcommand{\CV}{\mathcal{C}_V}

\newcommand{\PV}{\mathcal{P}^+(V)}
\newcommand{\PVV}{\mathcal{P}^-(V)}
\newcommand{\CT}{\mathcal{C}_{T}}

\newcommand{\la}{\lambda}

\makeatletter
\newcommand{\eqnum}{\refstepcounter{equation}\textup{\tagform@{\theequation}}}
\makeatother
\begin{document}

\begin{frontmatter}

\title{Linear pencils and quadratic programming problems with a quadratic constraint}

\author[IAM,UBA]{Santiago Gonzalez Zerbo} 
\ead{sgzerbo@fi.uba.ar}

\author[IAM,UBA]{Alejandra Maestripieri}
\ead{amaestri@fi.uba.ar}

\author[IAM,UNLP]{Francisco Mart\'{\i}nez Per\'{\i}a}
\ead{francisco@mate.unlp.edu.ar}

\address[IAM]{Instituto Argentino de Matem\'atica ``Alberto P. Calder\'on''\\ Saavedra 15, Piso 3 (1083) Buenos Aires, Argentina}
\address[UBA]{Departamento de Matem\'atica-- Facultad de Ingenier\'{\i}a -- Universidad de Buenos Aires\\ Paseo Col\'on 850 (1063) Buenos Aires, Argentina}
\address[UNLP]{Centro de Matem\'atica de La Plata -- Fac. de Cs. Exactas -- Universidad Nacional de La Plata\\ CC 172 (1900) La Plata, Argentina}

\begin{abstract}
Given bounded selfadjoint operators $A$ and $B$ acting on a Hilbert space $\HH$, consider the linear pencil $P(\la)=A+\la B$, $\la\in\RR$.
The set of parameters $\la$ such that $P(\la)$ is a positive (semi)definite operator is characterized. These results are applied to solving
a quadratic programming problem with an equality quadratic constraint (or a QP1EQC problem).

\end{abstract}


\end{frontmatter}

\section{Introduction}

The ideas considered in this paper are mainly motivated by the following quadratic programming problem with an equality
quadratic constraint (QP1EQC):
\begin{prob}\label{pb 0}
Given a Hilbert space $(\HH,\PI{\cdot}{\cdot})$, let $A$ and $B$ be bounded selfadjoint operators on $\HH$. Given $(w_0,z_0)\in\HH\x\HH$, analyze the existence of
\begin{equation*}
\min \PI{A(x-w_0)}{x-w_0} \quad \text{subject to} \quad \PI{B(x-z_0)}{x-z_0}=0,
\end{equation*}
and if the minimum exists, find the set of arguments at which it is attained.
\end{prob}

If $A$ and $B$ are indefinite (i.e. neither positive nor negative semidefinite) operators, this is a problem
where the objective function $x\mapsto \PI{A(x-w_0)}{x-w_0}$ is not convex while the function defining the constraint $x\mapsto\PI{B(x-z_0)}{x-z_0}$ is sign indefinite.  Quadratic programming (QP) with a
convex objective function was shown to be polynomial-time solvable. However, QP with an indefinite quadratic term is NP-hard in general. There is a lot of
literature on quadratically constrained quadratic programming (QCQP) problems, specially in the finite dimensional setting
\cite{Powell,Ye,Polik,Pencils,Park}.

The standard optimization technique, known as ``Lagrange multipliers'', consists in calculating the stationary points of the Fréchet derivatives in order to get
candidates for the solutions. In our particular case, a necessary condition for $x_0\in\HH$ to be a solution to Problem \ref{pb 0} is the existence of
$\la_0\in\RR$ such that 
\begin{align*}
(A+ \la_0 B)x_0 = Aw_0 + \la_0 Bz_0 \qquad \text{and} \qquad 
A + \la_0 B \ \text{is positive semidefinite},
\end{align*}
see e.g. \cite[\S 7.7 Thm. 2]{Luenberger} and \cite[Prop. 2.4.19]{Mardsen}.
Hence, the linear pencil $P(\la)=A+\la B$, $\la\in\RR$, is closely related to the above problem.
% Moreover, if $\la_0$ is such that $A+\la_0 B$ is positive
% definite then the pencil $P$ is regular and there are solutions to Problem \ref{pb 0} for each initial condition $(w_0,z_0)\in\HH\x\HH$.

Therefore, we are interested in describing the conditions under which the sets below are non empty:
\begin{align}\label{conjuntos}
I_\geq (A,B) &=\{\ \rho\in\RR:\ A + \rho B\ \text{is positive semidefinite}\ \}, \\
I_> (A,B) &=\{\ \rho\in\RR:\ A + \rho B\ \text{is positive definite}\ \}. \nonumber
\end{align}
The study of necessary and sufficient conditions for the existence
of a positive (semi)definite operator in the range of $P(\la)$ has attracted a lot of interest in the finite-dimensional
case, see \cite{BB71,BY95, L02,LLB13, KV95,NY91,NY93, NV03} among others. In particular, some of the results in Section \ref{Pencils} are inspired by \cite{HLS14}.

A large number of spectral problems for operator pencils of the form $P(\la)=A + \la B$ arises in different areas of applied mathematics, where $A$ and $B$ are symmetric (or selfadjoint) operators acting in a suitable Hilbert space $\HH$, see e.g. \cite{App3, App1, App2, App4, N16} and the references therein. Moreover, there is a well developed spectral theory for operator pencils with several applications, we only mention here the classical textbooks \cite{L1, L2,Gohberg}.

Our ultimate objective is to apply the results obtained in Section \ref{Pencils} to the regularization of the following indefinite least squares problem (ILSP) with an equality quadratic constraint:

\begin{prob}\label{pb 1}
Given a Hilbert space $\HH$ and Krein spaces $(\KK,\K{\cdot}{\cdot}_\KK)$ and $(\EE,\K{\cdot}{\cdot}_{\EE})$,
let $T:\HH\ra\KK$ and $V:\HH\ra\EE$ be bounded operators. Assume that $T$ has closed range and $V$ is surjective. 
Given $(w_0,z_0)\in\KK\x\EE$, analyze the existence of
\begin{equation*}
\min\,\K{Tx-w_0}{Tx-w_0}_\KK,\textit{ \normalfont{subject to} }\K{Vx-z_0}{Vx-z_0}_{\EE}=0,
\end{equation*}
and if the minimum exists, find the set of arguments at which it is attained.
\end{prob}

The paper is organized as follows. Section \ref{preliminaries} contains some notations used along the paper as well as two classical results that are frequently used.

Given bounded selfadjoint operators $A$ and $B$ on $\HH$, in Section \ref{Pencils} we describe the sets $I_\geq(A,B)$ and $I_>(A,B)$ defined in \eqref{conjuntos}. If $Q(B)$ denotes the set of neutral vectors for the quadratic form induced by $B$, it is well-known that $I_\geq(A,B)\neq\varnothing$ if and only if 
\[
\PI{Ax}{x}\geq 0 \quad \text{for every $x\in Q(B)$},
\]
i.e. the quadratic form induced by $A$ is nonnegative on $Q(B)$. Moreover, $I_\geq(A,B)$ is a closed interval $[\la_-,\la_+]$ in $\RR$, where the boundary values $\la_-$ and $\la_+$ are given by
\begin{align*}
\lambda_-:= -\inf_{\{x\in\HH:\PI{Bx}{x}>0\}}\frac{\PI{Ax}{x}}{\PI{Bx}{x}} \quad \text{and}\quad \lambda_+:= -\sup_{\{x\in\HH:\PI{Bx}{x}<0\}}\frac{\PI{Ax}{x}}{\PI{Bx}{x}}. 
\end{align*}
$I_>(A,B)$ is trivially contained in $I_\geq(A,B)=[\la_-,\la_+]$. If $\la_-\neq\la_+$ and $I_>(A,B)$ is non empty we show that $I_\geq(A,B)=(\la_-,\la_+)$, see Theorem \ref{pencil positivo}. Also, we prove that a necessary and sufficient condition for $I_>(A,B)\neq\varnothing$ is that the quadratic form induced by $A$ is uniformly positive on $Q(B)$, see \eqref{sufi pencil posi}.

% This is a selfadjoint operator with respect to the inner product $\PIL{\cdot}{\cdot}_\rho$ given by
% \[
% \PIL{x}{y}_\rho=\PI{(A+\rho B)x}{y}, \qquad x,y\in\HH.
% \]
%This is a selfadjoint operator. If we also assume that $B$ has closed range, $G$ can be written as the difference $G=G_+-G_-$ of two positive definite operators acting on suitable Hilbert spaces $\HH_+$ and $\HH_-$, respectively. We calculate the norms $\|G_\pm\|$.\textcolor{red}{\**}
% and, we show that these norms are attained in the set of vectors where the quotient $-\frac{\PI{Ax}{x}}{\PI{Bx}{x}}$ attains it extremal values $\la_-$ and $\la_+$, see \eqref{Ms} and Proposition \ref{autoespacio}.

In Section \ref{QPQC} we apply the results obtained in Section \ref{Pencils} to the Tikhonov's regularization of Problem \ref{pb 1}.
Defining $L:\HH\ra\KK\x\EE$ by $Lx=(Tx,Vx)$, the regularized problem can be restated as calculating
\[
\min_{x\in\HH} \K{Lx-(w_0,z_0)}{Lx-(w_0,z_0)}_\rho,
\]
where $\K{\cdot}{\cdot}_\rho$ is an indefinite inner product in $\KK\x\EE$ depending on the chosen regularization parameter $\rho$.
Since it is the indefinite least-squares problem (without constraints) associated to the equation $Lx=(w_0,z_0)$, the existence of solutions to such problem can be characterized in terms of the range of $L$, see \cite{GMMP10_2,GMMP16}. 

After presenting some preliminaries on Krein spaces, in Subsection \ref{props rango de L} we characterize different properties of $R(L)$ in terms of the operator $T^\#T + \rho V^\#V$. Then, we consider the linear pencil $P(\la)=T^\#T + \la V^\#V$ and we analyze it in its range of positiveness $[\rho_-, \rho_+]$. Finally, we show that if $R(L)$ is a nondegenerate (resp. pseudo-regular, regular) subspace of $(\KK\x\EE,\K{\cdot}{\cdot}_{\rho_0})$ for some $\rho_0\in(\rho_-,\rho_+)$, then $R(L)$ has the same property as a  subspace of $(\KK\x\EE,\K{\cdot}{\cdot}_{\rho})$ for every $\rho\in(\rho_-,\rho_+)$.

%which is the indefinite least-squares problem (without constraints) associated to the equation $Lx=(w_0,z_0)$, see \eqref{met indef para KxE}, \eqref{L} and \eqref{ILSP}. After presenting some preliminaries on Krein spaces, in Subsection \ref{props rango de L} we characterize different properties of $R(L)$, viewed as a subspace
%of the Krein space $(\KK\x\EE,\K{\cdot}{\cdot}_\rho)$ in terms of the operator $T^\#T + \rho V^\#V$. Then, we consider the linear pencil $P(\la)=T^\#T + \la V^\#V$ and we analyze it in its range of positiveness $[\rho_-, \rho_+]$. Finally, we show that if $R(L)$ is a nondegenerate (resp. pseudo-regular, regular) subspace of $(\KK\x\EE,\K{\cdot}{\cdot}_{\rho_0})$ for some $\rho_0\in(\rho_-,\rho_+)$, then $R(L)$ has the same property as a  subspace of $(\KK\x\EE,\K{\cdot}{\cdot}_{\rho})$ for every $\rho\in(\rho_-,\rho_+)$.  

\section{Preliminaries}\label{preliminaries}

Along this work $\HH$ denotes a complex (separable) Hilbert space. If $\mc{K}$ is another Hilbert space then $\mc{L}(\HH, \KK)$ is the vector space of bounded linear operators from $\HH$ into $\KK$ and $\mc{L}(\HH)=\mc{L}(\HH,\HH)$ stands for the algebra of bounded linear operators in $\HH$.

If $A\in \mc{L}(\HH, \KK)$ then $R(A)$ stands for the range of $A$ and $N(A)$ for its nullspace.
The next well-known result about the product of closed range operators is frequently used along the paper, see e.g. \cite{Izumino}.

\begin{prop}\label{R(AB)}
Given Hilbert spaces $\HH_1$, $\HH_2$ and $\KK$, let $A\in \mc{L}(\KK,\HH_2)$ and $B\in \mc{L}(\HH_1,\KK)$ have closed ranges. Then, $AB\in \mc{L}(\HH_1,\HH_2)$ has closed range if and only if $R(B) + N(A)$ is closed in $\KK$. %$c(R(B), N(A))<1$.
\end{prop}

An operator $A\in \mc{L}(\HH)$ is positive semidefinite if $\PI{Ax}{x}\geq 0$ for all $x\in\HH$; and it is positive definite if there exists $\alpha>0$ such that $\PI{Ax}{x}\geq \alpha\|x\|^2$ for every $x\in\HH$. The cone of positive semidefinite operators is denoted by $\mc{L}(\HH)^+$, and $GL(\HH)^+$ stands for the open cone of positive definite operators. 

We consider the order induced by $\mc{L}(\HH)^+$ into the real vector space of selfadjoint operators, known as L\"owner's order. If $A, B\in \mc{L}(\HH)$ are selfadjoint operators, $A\geq B$ stands for $A-B\in \mc{L}(\HH)^+$. In particular, $A\geq 0$ if and only if $A\in \mc{L}(\HH)^+$.
We say that a selfadjoint operator $A\in\mc{L}(\HH)$ is indefinite if it is neither positive nor negative semidefinite, i.e. if there exist
$x_+,x_-\in\HH$ such that $\PI{Ax_+}{x_+}>0$ and $\PI{Ax_-}{x_-}<0$.

The following result, due to R.~G.~Douglas \cite{Douglas}, characterizes
operator range inclusions. %It is frequently used along the paper.

\begin{thm}\label{Doug}
Given Hilbert spaces $\HH$, $\mc{K}_1$, $\mc{K}_2$ and operators
$A\in \mc{L}(\mc{K}_1,\HH)$ and $B\in \mc{L}(\mc{K}_2,\HH)$, the following
conditions are equivalent:
   \begin{enumerate}[label=\roman*)]
\item the equation $AX=B$ has a solution in
$\mc{L}(\mc{K}_2,\mc{K}_1)$;

\item $R(B)\subseteq R(A)$;

\item there exists $\lambda>0$ such that $BB^*\leq \lambda AA^*$.
\end{enumerate}
%In this case, there exists a unique $D\in \mc{L}(\mc{K}_2,\mc{K}_1)$
%such that $AD=B$ and $R(D)\subseteq \ol{R(A^*)}$.
%; moreover,
%$N(D)=N(B)$ and $\|D\|=\inf \{\lambda>0 : BB^*\leq \lambda
%AA^*\}$. The operator $D$ is called \emph{the reduced solution} of
%$AX=B$.
\end{thm}

An immediate consequence of this theorem is that $R(A)\subseteq R(A^{1/2})\subseteq \ol{R(A)}$ for $A\in \mc{L}(\HH)^+$.
Hence, $R(A)$ is closed if and only if $R(A)=R(A^{1/2})$.

\section{The range of positiveness of a linear operator pencil}\label{Pencils}

In this section we consider linear pencils of the form $P(\la)=A + \la B$, where $A$ and $B$ are bounded selfadjoint operators acting on a Hilbert space $(\HH,\PI{\cdot}{\cdot})$, and the parameter $\la$ runs through $\RR$. Inspired by \cite{HLS14}, we are interested in describing the conditions under which the sets below are non empty:
\begin{align*}
I_\geq (A,B) &=\{\rho\in\RR:\ A + \rho B\in \mc{L}(\HH)^+\}, \\
I_> (A,B) &=\{\rho\in\RR:\ A + \rho B\in GL(\HH)^+\}.
\end{align*}
%The study of necessary and sufficient conditions which guarantee the existence
%of a positive (semi)definite operator in the range of $P(\la)$ has attracted a lot of interest in the finite-dimensional
%case, see \cite{BB71,BY95, L02,LLB13, KV95,NY91,NY93} among others.
This problem has a long history,  which is thoroughly reviewed in \cite{U79}. For matrix pencils it can be traced back to \cite{F,BB71}, and in the case of operator pencils in Hilbert spaces to \cite{C64,K,KS}. 
%Our ultimate objective is to apply these results to the study of Problem \ref{pb 1}, a quadratic programming problem with one quadratic equality constraint.     

\smallskip

Recall the following theorem, taken from \cite{Azizov}, which is a version of
\cite[Thm. 1.1]{KS}. Denote by $Q(B)$ the set of neutral vectors for the quadratic form induced by the selfadjoint operator $B$:
\[
Q(B)=\{x\in\HH: \ \PI{Bx}{x}=0\}.
\]

\begin{thm}\label{Azizov para pencils}
Let $A,B\in \mc{L}(\HH)$ be selfadjoint operators and suppose  that $B$ is indefinite.
%i.e. there exist $y_0,z_0\in \HH$ such that $\PI{By_0}{y_0}<0$ and $\PI{Bz_0}{z_0}>0$.
Also, assume that 
\begin{equation}\label{sufi pencil noneg}
\PI{Ax}{x}\geq 0 \qquad \text{for every $x\in Q(B)$}.%$x\in\HH$ such that $\PI{Bx}{x}=0$}. 
\end{equation}
Then, for every $y\in\HH$ such that $\PI{By}{y}<0$ and for every $z\in\HH$ such that $\PI{Bz}{z}>0$ we have that
\[
\frac{\PI{Ay}{y}}{\PI{By}{y}}\leq \frac{\PI{Az}{z}}{\PI{Bz}{z}}.
\]
Hence,
\begin{align}\label{mus}
\mu_+= &\inf_{ \{z\in\HH: \PI{Bz}{z}=1\} } \PI{Az}{z}>-\infty \quad \text{and} \\ \mu_-= &\sup_{ \{y\in\HH: \PI{By}{y}=-1\} } \PI{Ax}{x}<+\infty. \nonumber
\end{align}
Moreover, $\mu_-\leq \mu_+$ and for any $\mu\in[\mu_-, \mu_+]$ the following inequality holds:
\[
\PI{Ax}{x}\geq \mu \PI{Bx}{x} \qquad \text{for every $x\in\HH$}.
\]
\end{thm}

This implies that \eqref{sufi pencil noneg} is a sufficient condition to guarantee that $I_\geq (A,B)\neq\varnothing$, but it is also a necessary one. The next result is a generalization of \cite[Thm.2]{HLS14}, see also \cite{hamburguer,More}.

\begin{prop}\label{Azizov pencils}
Let $A,B\in \mc{L}(\HH)$ be selfadjoint operators and suppose  that $B$ is indefinite. Then, $I_\geq (A,B)\neq\varnothing$ if and only if 
\[
\PI{Ax}{x}\geq 0 \qquad \text{for every $x\in Q(B)$}. %\HH$ such that $\PI{Bx}{x}=0$}. 
\]
In this case, there exist $\lambda_-,\lambda_+\in\RR$ such that $\lambda_-\leq\lambda_+$ and $I_\geq (A,B)=[\lambda_-,\lambda_+]$.
\end{prop}

\begin{proof}
Assume that $I_\geq (A,B)\neq\varnothing$ and consider $\rho\in I_\geq (A,B)$. Given $x\in\HH$ such that $\PI{Bx}{x}=0$, we have
\[
0\leq \PI{(A + \rho B)x}{x}=\PI{Ax}{x} + \rho\PI{Bx}{x}=\PI{Ax}{x}.
\]
Therefore, \eqref{sufi pencil noneg} holds. The converse follows from Theorem \ref{Azizov para pencils}, as well as the description of $I_\geq (A,B)$ as an interval. In fact, as a consequence of \eqref{mus}, the constants $\lambda_-$ and $\lambda_+$ are given by
\[
\lambda_-:=-\mu_+ \quad \text{and}\quad\lambda_+:=-\mu_-. \qedhere
\]
%\inf_{\{x\in\HH:\PI{Bx}{x}>0\}}\frac{\PI{Ax}{x}}{\PI{Bx}{x}}=-
%-\sup_{\{x\in\HH:\PI{Bx}{x}<0\}}\frac{\PI{Ax}{x}}{\PI{Bx}{x}}=
\end{proof}

The first assertion in \cite[Thm.\,3]{HLS14} holds true for (bounded) selfadjoint operators acting on a Hilbert space $(\HH,\PI{\cdot}{\cdot})$, i.e. if $I_\geq(A,B)=\{\lambda_0\}$ then $B$ is indefinite and $I_>(A,B)=\varnothing$. The proof is exactly the same presented there. 
The last part of \cite[Thm.3]{HLS14} can be generalized as follows. 

\begin{prop}\label{nucleo del interior}
Let $A,B\in \mc{L}(\HH)$ be selfadjoint operators and assume that $I_\geq(A,B)=[\lambda_-,\lambda_+]$ with $\lambda_-<\lambda_+$. Then,
\begin{enumerate}[label=\roman*)]
	\item  $N(A + \la B)=Q(A)\cap Q(B)=N(A)\cap N(B)$ for every $\la\in (\lambda_-,\lambda_+)$.
	\item Assume $I_>(A,B)\neq\varnothing$. Given $\la\in I_>(A,B)$, there exist $S\in GL(\HH)^+$, a compact subset $K$ of $\RR$ and a spectral measure $\mu$ defined on $K$ such that
	\begin{equation}\label{desc espect}
	SBS=\int_K t\ d\mu(t) \qquad \text{and} \qquad SAS=\int_K (1-\la t)\ d\mu(t).
	\end{equation}
\end{enumerate}
\end{prop}

\begin{proof}
Assume that $\la_-\neq\la_+$ and consider $\la\in (\lambda_-,\lambda_+)$. Since $A+\la B\in \mc{L}(\HH)^+$, it is easy to see that 
\[
Q(A)\cap Q(B)= N(A+\la B)\cap Q(B).
\]
Now, assume that there exists $x_0\in N(A+\la B)\setminus(Q(A)\cap Q(B))=
N(A+\la B)\setminus Q(B)$. Since $\PI{Ax_0}{x_0}+\la\PI{Bx_0}{x_0}=0$ and $\la_-<\la<\la_+$,
\begin{equation}\label{algo_1}
\sup_{\{x\in\HH:\PI{Bx}{x}<0\}}\frac{\PI{Ax}{x}}{\PI{Bx}{x}}=-\la_+<\frac{\PI{Ax_0}{x_0}}{\PI{Bx_0}{x_0}}<-\la_-=\inf_{\{x\in\HH:\PI{Bx}{x}>0\}}\frac{\PI{Ax}{x}}{\PI{Bx}{x}}.
\end{equation}
Since $\PI{Bx_0}{x_0}>0$ or $\PI{Bx_0}{x_0}<0$, \eqref{algo_1} leads to a contradiction. Therefore, 
\[
N(A+\la B)=Q(A)\cap Q(B).
\]
In particular, $N(A+\la B)$ does not depend on $\la$. Hence,  if $x\in Q(A)\cap Q(B)$ then $Ax=-\la' Bx$ for any $\la'\in(\la_-,\la_+)$.
This implies that if $x\in N(A+\la B)$ then $x\in N(A)\cap N(B)$, completing the proof of item {\it i}.

 To prove item {\it ii}, consider a fixed $\la\in I_>(A,B)$ and denote $S:=(A+\la B)^{-1/2}\in GL(\HH)^+$. Then, 
\[
SAS =I- \la SBS.
\]
Hence, if $K=\sigma(SBS)$ and $\mu$ is the spectral measure of $SBS$, then \eqref{desc espect} holds. 
\end{proof}

Item {\it i} can also be proven following the same lines as in the proof of  \cite[Thm.\,3]{HLS14}. Item {\it ii} in the above proposition  can be read as a ``simultaneous diagonalization via congruence'' for selfadjoint operators acting on a Hilbert space.

\medskip

%A similar result holds for the set $I_> (A,B)$.
Our aim is to show that if $I_\geq(A,B)=[\la_-,\la_+]$ and $\la_-<\la_+$, then $I_>(A,B)=(\la_-,\la_+)$. First, we need to prove some technical preliminaries. In the first place, for $\lambda\in[\lambda_-,\lambda_+]$ consider the seminorm $\|\cdot\|_\lambda$ defined by
\[
\|x\|_\lambda=\PI{(A+\lambda B)x}{x}^{1/2},\quad\quad x\in\HH.
\]
By Proposition \ref{nucleo del interior}, if $\la\in(\la_-,\la_+)$ then $\|\cdot\|_{\la}$ is a norm in the Hilbert space $\HH':=\big(N(A)\cap N(B)\big)^\bot$.

\begin{lema}\label{lema:normas_equivalentes_p} 
If $\la_-\neq\la_+$ then $\|\cdot\|_{\la}$ and $\|\cdot\|_{\la'}$ are equivalent on $\HH'$
for every pair $\la,\la'\in(\la_-,\la_+)$.
\end{lema}

%\smallskip

\begin{proof}
Given $\la,\la'\in(\la_-,\la_+)$ assume that $\la'>\la$. For an arbitrary $x\in\HH'$, $x\neq 0$, 
\begin{align*}
\|x\|_{\la'}^2&=\PI{Ax}{x}+\la'\PI{Bx}{x}=\parentesis{\PI{Ax}{x}+\la \PI{Bx}{x}}+(\la'-\la)\PI{Bx}{x}\\
&=\|x\|_{\la}^2+(\la'-\la)\PI{Bx}{x},
\end{align*}
and consequently,
\begin{equation}\label{eq_1_p}
\frac{\|x\|_{\la'}^2}{\|x\|_{\la}^2}=1+(\la'-\la)\frac{\PI{Bx}{x}}{\|x\|^2_{\la}}=1+(\la'-\la)\frac{\PI{Bx}{x}}{\PI{Ax}{x}+\la\PI{Bx}{x}}.
\end{equation}
%We shall see that
Now we show that
\begin{equation}\label{eq:final_desigualdad_p}
-\frac{1}{\la_+-\la}\leq\frac{\PI{Bx}{x}}{\PI{Ax}{x}+\la\PI{Bx}{x}}\leq\frac{1}{\la-\la_-}.
\end{equation}
If $\PI{Bx}{x}=0$, \eqref{eq:final_desigualdad_p} is trivially satisfied. If $\PI{Bx}{x}>0$, then
\begin{align*}
\frac{\PI{Bx}{x}}{\PI{Ax}{x}+\la\PI{Bx}{x}}&\leq\sup_{\{y\in\HH:\ \PI{By}{y}>0\}}\frac{\PI{By}{y}}{\PI{Ay}{y}+\la\PI{By}{y}} \\ &=\frac{1}{\displaystyle{\inf_{\{y\in\HH:\ \PI{By}{y}>0\}}\frac{\PI{Ay}{y}}{\PI{By}{y}}}+\la}=\frac{1}{\la-\la_-}.
\end{align*}
Furthermore, since $\PI{Ax}{x}+\la\PI{Bx}{x}>0$ and $\PI{Bx}{x}>0$, % and $\rho_+>\rho$,
\[
-\frac{1}{\la_+-\la}< 0< \frac{\PI{Bx}{x}}{\PI{Ax}{x}+\la\PI{Bx}{x}}.
\]
A similar argument can be given if $\PI{Bx}{x}<0$.

As a consequence of \eqref{eq:final_desigualdad_p} we have that
\begin{equation}\label{eq_2_p}
\frac{\la_+-\la'}{\la_+-\la}\leq 1+(\la'-\la)\frac{\PI{Bx}{x}}{\PI{Ax}{x}+\la\PI{Bx}{x}}\leq \frac{\la'-\la_-}{\la-\la_-}.
\end{equation}
%1-\frac{\la'-\la}{\la_+-\la} 1+\frac{\la'-\la}{\la-\la_-}=
By \eqref{eq_1_p} and \eqref{eq_2_p},
\[
\parentesis{\frac{\la_+-\la'}{\la_+-\la}}^{1/2}\leq\frac{\|x\|_{\la'}}{\|x\|_{\la}}\leq\parentesis{\frac{\la'-\la_-}{\la-\la_-}}^{1/2},
\]
and consequently, the norms $\|\cdot\|_{\la'}$ and $\|\cdot\|_\la$ are equivalent.
\end{proof}

\begin{obs}\label{cuasi equiv} 
With the same procedure used in the proof of Lemma \ref{lema:normas_equivalentes_p}, it can be shown that for every $\la\in(\la_-,\la_+)$ and every
$x\in\HH'$
\begin{equation}\label{eq:normas_extremos}
\|x\|_{\la_-}\leq\parentesis{\frac{\la_+-\la_-}{\la_+-\la}}^{1/2}\|x\|_\la\quad\text{and}\quad\|x\|_{\la_+}\leq\parentesis{\frac{\la_+-\la_-}{\la-\la_-}}^{1/2}\|x\|_\la.
\end{equation}
\end{obs}

\begin{cor}\label{rango raiz}
Let $A,B\in \mc{L}(\HH)$ be selfadjoint operators and suppose  that $B$ is indefinite. Assume $I_\geq(A,B)=[\lambda_-,\lambda_+]$ with $\lambda_-<\lambda_+$. Then,
   \begin{enumerate}[label=\roman*)]
\item $R\big((A + \la B)^{1/2}\big)=R\big((A + \la' B)^{1/2}\big)$ for every $\la,\la'\in (\la_-,\la_+)$; 
\item $R\big((A + \la_\pm B)^{1/2}\big)\subseteq R\big((A + \la B)^{1/2}\big)$ for every $\la\in (\la_-,\la_+)$. 
\end{enumerate}
\end{cor}

\begin{proof}
By Proposition \ref{nucleo del interior}, we have that $N(A+ \la B)^\bot=\HH'$ for every $\la\in (\la_-,\la_+)$. Then, given $\la,\la'\in (\la_-,\la_+)$,
$\la'\neq\la$, the norm equivalence in Lemma \ref{lema:normas_equivalentes_p} can be rephrased as: there exist $0<\alpha<\beta$ such that 
\[
\alpha(A+\lambda' B)\leq A+\lambda B\leq\beta(A+\lambda' B).
\]
% \[
% \alpha \PI{(A+\la' B)x}{x} \leq \PI{(A + \la B)x}{x} \leq \beta \PI{(A+ \la' B)x}{x} \qquad \text{for every $x\in \HH'$}.
% \] 
%Since $A +\la B=A + \la'B=0$ on $(\HH')^\bot$, we have that $\alpha (A+\la' B) \leq A+\la B \leq \beta (A+\la' B)$.
Applying Theorem \ref{Doug}, 
\[
R\big((A + \la B)^{1/2}\big)=R\big((A + \la' B)^{1/2}\big).
\]

Analogously, the inequalities in \eqref{eq:normas_extremos} imply there exist $\alpha_\pm>0$ such that
\[
A+\lambda_\pm B\leq \alpha_\pm(A+\lambda B).
\]
Hence,
\[
R\big((A + \la_\pm B)^{1/2}\big)\subseteq R\big((A + \la B)^{1/2}\big). \qedhere
\]
\end{proof}

%Although the range of the square-root remains invariant, it is possible that $R(A+\la B)\neq R(A + \la'B)$ for $\la,\la'\in (\la_-,\la_+)$, $\la'\neq\la$. However, since $R(A+\la B)\subset R((A+\la B)^{1/2})\subseteq \ol{R(A+\la B)}$ for every $\la\in (\la_-,\la_+)$, the closure of the range also remains invariant. In particular,

\begin{cor}\label{rangos cerrados} 
Let $A,B\in \mc{L}(\HH)$ be selfadjoint operators and suppose  that $B$ is indefinite. Assume $I_\geq(A,B)=[\lambda_-,\lambda_+]$ with $\lambda_-<\lambda_+$. If $A+\lambda_0 B$ has closed range for some $\lambda_0\in(\lambda_-,\lambda_+)$, then $A+ \lambda B$ has closed range for every $\lambda \in(\lambda_-,\lambda_+)$.
\end{cor}

%\begin{proof} 
%If $R(A+\lambda_0 B)$ is closed in $(\HH,\PI{\cdot}{\cdot})$, then $\|\cdot\|$ and $\|\cdot\|_{\lambda_0}$ are equivalent in $\HH'$. Given $\lambda\in(\lambda_-,\lambda_+)$,
%by Lemma \ref{lema:normas_equivalentes_p}, $\|\cdot\|_\lambda$ and $\|\cdot\|_{\lambda_0}$ are also equivalent on $\HH'$. Therefore, $\|\cdot\|_\lambda$ and $\|\cdot\|$ are equivalent on $\HH'$.
%Hence, there exists $M>0$ such that
%\[
%M\|x\|^2\leq\|x\|_{\lambda}^2=\|(A+\lambda B)^{1/2}x\|^2,\quad\quad\text{for every $x\in \HH'$}.
%\]
%Thus, $R((A+\lambda B)^{1/2})=R(A+\lambda B)$ is a closed subspace of $\HH$.
%\end{proof}

The main result of this section shows that, if $I_> (A,B)$ is not empty then $I_> (A,B)$ coincides with the interior of $I_\geq (A,B)$. 

\begin{thm}\label{pencil positivo}
Let $A,B\in \mc{L}(\HH)$ be selfadjoint operators and suppose that $B$ is indefinite.
Assume $I_\geq(A,B)=[\lambda_-,\lambda_+]$ with $\lambda_-\leq\lambda_+$.
 Then, $I_> (A,B)\neq\varnothing$ if and only if there exists $\alpha>0$ such that  
\begin{equation}\label{sufi pencil posi}
\PI{Ax}{x}\geq \alpha\|x\|^2 \qquad \text{for every $x\in\HH$ such that $\PI{Bx}{x}=0$}. 
\end{equation}
In this case, $\la_-<\la_+$ and $I_> (A,B)=(\lambda_-,\lambda_+)$.
% Moreover, for afixed $\alpha>0$ such that \eqref{sufi pencil posi} holds,
% there exist $\eta_-,\eta_+\in\RR$ such that $\eta_-\leq\eta_+$ and $\{\la\in\RR:\ A+\la B\geq\alpha I\}=[\eta_-,\eta_+]$.
\end{thm}

\begin{proof}
Assume that $I_> (A,B)\neq\varnothing$ and consider $\rho\in I_> (A,B)$. Then, $A+\rho B$ is injective and, by Proposition \ref{nucleo del interior}, $N(A+ \la B)=\{0\}$ for every $\la\in (\la_-,\la_+)$. Also,  $A+\rho B$ has closed range and, by Proposition \ref{rangos cerrados}, $R(A+ \la B)$ is closed for every $\la\in (\la_-,\la_+)$. Therefore, $(\la_-,\la_+)\subseteq I_>(A,B)$. 

Also, since $A+\rho B$ is invertible, there exists $\alpha>0$ such that $\PI{(A + \rho B)x}{x}\geq \alpha\|x\|^2$ for every $x\in\HH$. Hence, given $x\in\HH$ such that $\PI{Bx}{x}=0$ we have
\[
\alpha\|x\|^2\leq \PI{(A + \rho B)x}{x}=\PI{Ax}{x},
\]
which proves \eqref{sufi pencil posi}.

Now, let us show that $\la_\pm\notin I_>(A,B)$. If $\la_\pm\in I_>(A,B)$, then 
%the same argument used above shows that
%$A+\la_\pm B$ is (injective and) bounded from below. Hence, 
there exists $\alpha_\pm>0$ such that $\PI{(A + \la_\pm B)x}{x}\geq \alpha_\pm\|x\|^2$ for every $x\in\HH$. Since $|\PI{Bx}{x}|\leq \|B\|\ \|x\|^2$ for all $x\in\HH$, we have that
\begin{align*}
\PI{\left(A + \big(\la_\pm \pm\tfrac{\alpha_\pm}{\|B\|}\big)B\right)x}{x} &= \PI{(A+\la_\pm B)x}{x} \pm \frac{\alpha_\pm}{\|B\|}\PI{Bx}{x} \\ &\geq \PI{(A + \la_\pm B)x}{x}- \alpha_\pm\|x\|^2\geq 0,
\end{align*}
for every $x\in\HH$, i.e. $\la_\pm \pm\tfrac{\alpha_\pm}{\|B\|}\in I_\geq(A,B)$, which is a contradiction. Thus, $I_>(A,B)=(\la_-,\la_+)$.

Conversely, assume that \eqref{sufi pencil posi} holds for a given $\alpha>0$. Then, the selfadjoint operators $A':=A-\alpha I$
and $B$ satisfy \eqref{sufi pencil noneg}. Thus, there exist $\eta_-,\eta_+\in\RR$ such that
$\{\la\in\RR:\ A+\la B\geq\alpha I\}=I_\geq (A',B)=[\eta_-,\eta_+]$.
If $\rho\in[\eta_-,\eta_+]$ then $A+\rho B$ is a positive definite operator, i.e. $\rho\in I_>(A,B)$.  
\end{proof}

To conclude this section, we characterize the nullspaces of the operators associated to the extreme values of $I_\geq(A,B)$, i.e. $N(A+\la_- B)$ and $N(A+\la_+ B)$. To do so, we introduce the following sets:
\begin{align}\label{Ms}
\mc{M}_\pm &:=\displaystyle{\set{x\in\HH:\ \PI{Bx}{x}\neq0 \ \ \text{and}\ \ \frac{\PI{Ax}{x}}{\PI{Bx}{x}}=-\la_\mp}}. %\quad \text{and}  \\ \quad\mc{M}_- &:=\displaystyle{\set{x\in\HH:\ \PI{Bx}{x}<0 \ \ \text{and}\ \ \frac{\PI{Ax}{x}}{\PI{Bx}{x}}=-\la_+}}. \nonumber
\end{align}

\begin{obs}
If $\la_-< \la_+$ then 
\[
\mc{M}_+ =\displaystyle{\set{x\in\HH:\ \PI{Bx}{x}>0 \ \ \text{and}\ \ \frac{\PI{Ax}{x}}{\PI{Bx}{x}}=-\la_-}},
\]
and
\[
\mc{M}_- =\displaystyle{\set{x\in\HH:\ \PI{Bx}{x}<0 \ \ \text{and}\ \ \frac{\PI{Ax}{x}}{\PI{Bx}{x}}=-\la_+}}.
\] 
Indeed, assume that there exists $x_0\in\HH$ such that $\PI{Bx_0}{x_0}<0$ and $x_0\in\M_+$. Then,
\[
-\la_+ = \sup_{\{x\in\HH:\PI{Bx}{x}<0\}}\frac{\PI{Ax}{x}}{\PI{Bx}{x}}\geq \frac{\PI{Ax_0}{x_0}}{\PI{Bx_0}{x_0}}=-\la_-,
\]
which is a contradiction to $\la_-<\la_+$. A similar argument holds to show the alternative description of $\M_-$.
\end{obs}

%Although $\M_\pm$ is not a subspace, if $\M_\pm\neq\varnothing$ then $\M_\pm + N(A)\cap N(B)=\M_\pm$.

\begin{lema}\label{lema:nucleo_lambda+}
Let $A,B\in \mc{L}(\HH)$ be selfadjoint operators and suppose  that $B$ is indefinite.    \begin{enumerate}[label=\roman*)]
\item[i)] if $\la_-\neq\la_+$ then $N(A+\la_\pm B)=\mc{M}_\mp\cup \parentesis{N(A)\cap N(B)}$.
\item[ii)] If $\la_-=\la_+$ then $\M_-=\M_+$ and $N(A+\la_\pm B)=\mc{M}_+\cup \parentesis{Q(A)\cap Q(B)}$.
\end{enumerate}
\end{lema}

\begin{proof}
First assume that $\la_-\neq\la_+$, and consider the left boundary $\la_-$ (similar considerations hold for the right one, $\la_+$). Since $A+\la_-B\in\mc{L}(\HH)^+$, we have that $x\in N(A+\la_-B)$ if and only if $\PI{Ax}{x}=-\la_-\PI{Bx}{x}$.

\smallskip

%In the first place, it is trivial that $N(A)\cap N(B)\subseteq N(A+\la_- B)$.
For $x\in\mc{M}_+$ or $x\in N(A)\cap N(B)$,
$\PI{Ax}{x}=-\la_-\PI{Bx}{x}$ holds,
thus the inclusion $\mc{M}_+\cup\big(N(A)\cap N(B)\big)\subseteq N(A+\la_-B)$ is immediate. 

Now, let $x\in N(A+\la_- B)$. On the one hand,
if $\PI{Bx}{x}=0$ then $\PI{Ax}{x}=-\la_-\PI{Bx}{x}=0$ and, by Lemma \ref{nucleo del interior}, $x\in Q(A)\cap Q(B)=N(A)\cap N(B)$. On the other hand, if $\PI{Bx}{x}\neq 0$ then
\[
\frac{\PI{Ax}{x}}{\PI{Bx}{x}}=-\la_-,
\]
so that $x\in\mc{M}_+$. 
%Finally, assume that $\PI{Bx}{x}<0$. For any $\la\in(\la_-,\la_+)$, we have that
%$x\notin N(A+\la B)=N(A)\cap N(B)$, and thus $\PI{Ax}{x}+\la \PI{Bx}{x}=\PI{(A+\la B)x}{x}>0$.
%Then,
%\[
%\frac{\PI{Ax}{x}}{\PI{Ax}{x}+\la\PI{Bx}{x}}=\frac{1}{\la+\displaystyle{\frac{\PI{Ax}{x}}{\PI{Bx}{x}}}}=\frac{1}{\la-\la_-}>0.
%\]
%Hence, $\PI{Ax}{x}>0$ and $\PI{Bx}{x}<0$, which is a contradiction to $x\in N(A+\la_- B)$. Hence, $N(A+\la_- B)\subseteq \mc{M}_+\cup\parentesis{N(T^\#T)\cap N(V)}$,
%completing the proof.

\smallskip
%Now, assume that $\la_-=\la_+$ and denote $\la=\la_\pm$. Using the same procedure as above, the assertion follows immediately considering that $x\in N(A+\la B)$ if and only if $\PI{Ax}{x}=-\la\PI{Bx}{x}$. 

Using the same procedure as above, if we assume that $\la_-=\la_+$ the assertion follows immediately. 
\end{proof}

If $\HH$ is finite dimensional then $\M_+$ and $\M_-$ are trivially non empty. However, the next example shows that the set $\M_+$ may be empty in the infinite dimensional setting. A similar example can be constructed to show that $\M_-$ may also be empty.

\begin{example}
Let us consider $\HH=\ell^2(\NN)$ endowed with the usual inner product:
\[
\PI{(x_n)_{n\in\Z}}{(y_n)_{n\in\NN}}=\sum_{n\in\NN}x_n \ol{y_n}, \qquad (x_n)_{n\in\NN},(y_n)_{n\in\NN}\in \HH.
\]
If $A, B\in\mc{L}(\HH)$ are given by 
\begin{align*}
A(x_n)_{n\in\NN} &=(y_n)_{n\in\NN}, \qquad \text{with}\ \ y_1=x_1,\ \text{and}\ y_n= -\left(1-\tfrac{1}{n}\right)x_n \ \text{if $n\neq 1$}, \\
B(x_n)_{n\in\NN} &=(z_n)_{n\in\NN}, \qquad \text{with}\ \ z_1=-\tfrac{1}{2}x_1,\ \text{and}\ z_n= x_n \ \text{if $n\neq 1$}, 
\end{align*}
then both $A$ and $B$ are selfadjoint indefinite operators. Given $\la\in\RR$,
\[
\PI{(A+\la B)(x_n)_{n\in\NN}}{(x_n)_{n\in\NN}}=\big(1-\tfrac{1}{2}\la\big)|x_1|^2 + \sum_{n\neq 1}\left(\la-\big(1-\tfrac{1}{n}\big)\right)|x_n|^2.
\]
Hence, $A + \la B\in \mc{L}(\HH)^+$ if and only if $1-\tfrac{1}{2}\la\geq 0$ and $\la-\big(1-\tfrac{1}{n}\big)\geq 0$
for every $n\in\NN\setminus\set{1}$. Therefore,
\[
\la_-=1 \quad \text{and} \quad \la_+=2.
\]
Observe that $N(A + \la_- B)=\{0\}$ because
\[
\PI{(A+\la_- B)(x_n)_{n\in\NN}}{(x_n)_{n\in\NN}}=\tfrac{1}{2} |x_1|^2 + \sum_{n\neq 1}\tfrac{1}{n} |x_n|^2,
\]
and it is zero if and only if $x_n=0$ for every $n\in\NN$. Thus, $\M_+=\varnothing$.
\end{example}

\bigskip\bigskip

Assume that $I_>(A,B)\neq\varnothing$ and let $\rho\in I_>(A,B)$. 
Considering the selfadjoint operator $G\in\mc{L}(\HH)$ given by
\[
G:=(A+ \rho B)^{-1/2} B\,(A+\rho B)^{-1/2},
\]
the pencil $P(\lambda)=A+\la B$ is then congruent to $I+(\la-\rho)G$. Indeed,
\[
A + \la B=(A + \rho B)^{1/2}\big(I + (\la-\rho)G \big)(A+\rho B)^{1/2}.
\]
This reduction is frequently used in the operator pencils context, see e.g. \cite[Chapter IV]{Gohberg} and \cite{Hmam}.
We end this section by characterizing the range of positiveness of operator pencils of such form, when $G$ is an indefinite operator.
 
% \medskip

% Under the above assumption, $\HH'=\HH$ and the norms $\|\cdot\|$ and $\|\cdot\|_{\rho}$ are equivalent on $\HH$. Hence, $G$ is also bounded in the Hilbert space $(\HH,\PIL{\cdot}{\cdot}_{\rho})$, where
% \[
% \PIL{x}{y}_\rho=\PI{(A+\rho B)x}{y}, \qquad x,y\in\HH.
% \] 
% The advantage of considering this new inner product is that $G$ is selfadjoint with respect to $\PIL{\cdot}{\cdot}_\rho$: 
% \[
% \PIL{Gx}{y}_\rho=\PI{Bx}{y}=\PI{x}{By}=\PIL{x}{Gy}_\rho \qquad \text{for every $x,y\in \HH$}.
% \]

\medskip

Given an indefinite selfadjoint operator $G\in\mc{L}(\HH)$, consider its canonical decomposition as the difference of two positive operators, i.e. consider the decomposition 
\begin{equation}\label{eq:descomposicion}
\HH=\HH_+\oplus\HH_-\oplus N(G),
\end{equation}
%where $N(G)=N(V)$, 
and let $G_\pm\in \mc{L}(\HH_\pm)^+$ be such that $G=\left(\begin{smallmatrix}
G_+ & 0 & 0 \\
0 & -G_- & 0 \\
0 & 0 & 0
\end{smallmatrix}\right)$ with respect to \eqref{eq:descomposicion}.

%Then there exist two closed subspaces $\HH_+$ and $\HH_-$ and two invertible operators $G_\pm\in \mc{L}(\HH_\pm)^+$ such that $G=G_+-G_-$ and 
%\begin{equation}\label{eq:descomposicion}
%\HH=\HH_+\oplus_\rho\HH_-\oplus_\rho N(G),
%\end{equation}
%where $N(G)=N(V)$.

The following statement result describes the range of positiveness of the operator pencil $I+\eta G$ in terms of the the norms of the operators $G_+$ and $G_-$.

\begin{prop}
Given a selfadjoint operator $G\in\mc{L}(\HH)$, if $G$ is indefinite then
\[
I_\geq(I,G)=\big[-\|G_+\|^{-1},\|G_-\|^{-1}\big]\quad\quad\text{and}\quad\quad I_>(I,G)=\big(-\|G_+\|^{-1},\|G_-\|^{-1}\big).
\]
\end{prop}

\begin{proof}
By Proposition \ref{Azizov pencils} and Theorem \ref{pencil positivo}, $I_{\geq}(I,G)=[\lambda_-,\lambda_+]$ with $\lambda_-<\lambda_+$ and
$I_{>}(I,G)=(\lambda_-,\lambda_+)$.
Given $x\in\HH$, assume that it is decomposed as $x=x^+ + x^- + x^0$, where $x^\pm\in\HH_\pm$ and $x^0\in N(G)$. Also, assume that $x^+\neq 0$. Then,
\begin{equation}\label{eq:desigualdad_G}
\frac{\PI{Gx}{x}}{\|x\|^2}=\frac{\PI{G_+x^+}{x^+} - \PI{G_-x^-}{x^-} }{\|x^+\|^2 + \|x^-\|^2 + \|x^0\|^2} \leq \frac{\PI{G_+x^+}{x^+}}{\|x^+\|^2}=\frac{\PI{Gx^+}{x^+}}{\|x^+\|^2}.
\end{equation}
Since $\HH_+\setminus\set{0}\subseteq\set{x\in\HH\,:\,\PI{Gx}{x}>0}$,
%\begin{equation}\label{eq:desigualdad_G_2}
%\sup_{x\in\HH_+\setminus\set{0}}\frac{\PI{Gx}{x}}{\|x\|^2}\leq\sup_{x\in\HH:\PI{Gx}{x}>0}\frac{\PI{Gx}{x}}{\|x\|^2}.
%\end{equation}
%Combining \eqref{eq:desigualdad_G} and \eqref{eq:desigualdad_G_2}, it turns out
\eqref{eq:desigualdad_G} implies that
\[
\sup_{x\in\HH:\PI{Gx}{x}>0)}\frac{\PI{Gx}{x}}{\|x\|^2}=\sup_{x\in\HH_+\setminus\{0\}} \frac{\PI{G_+x}{x}}{\|x\|^2}=\|G_+\|<+\infty,
\]
and consequently,
\begin{equation*}
\lambda_-= -\inf_{x\in\HH:\PI{Gx}{x}>0}\frac{\PI{x}{x}}{\PI{Gx}{x}}=-\Bigg(\sup_{x\in\HH:\PI{Gx}{x}>0}\frac{\PI{Gx}{x}}{\|x\|^2}\Bigg)^{-1}=-\|G_+\|^{-1}.
\end{equation*}
%Therefore, $\rho>\rho_-$. 

% Following a similar procedure with vectors in $\HH_-$ we have that
% \[
% \inf_{\{x\in\HH:\PI{Bx}{x}<0\}}\frac{\PI{Gx}{x}}{\|x\|^2}=\inf_{x\in\HH_-\setminus\{0\}} \frac{-\PI{G_-x}{x}}{\|x\|^2}=-\|G_-\|>-\infty,
% \]
% and therefore $\|G_-\|=\frac{1}{\lambda_+-\rho}$.
A similar procedure with vectors in $\HH_-$ proves that $\lambda_+=\|G_-\|^{-1}$.
\end{proof}

\section{On the regularization of a QP1EQC}\label{QPQC}

Along this section we apply the results obtained before to the regularization of the indefinite least-squares problem with a
quadratic constraint presented as  Problem \ref{pb 1}. Hereafter, given a Hilbert space $(\HH,\PI{\cdot}{\cdot})$, and
Krein spaces $(\KK,\K{\cdot}{\cdot}_\KK)$ and $(\EE,\K{\cdot}{\cdot}_{\EE})$,
let $T\in \mc{L}(\HH,\KK)$ and $V\in \mc{L}(\HH,\EE)$ be closed range operators.

\medskip

Given a real constant $\rho\neq0$, the regularization of Problem \ref{pb 1} consists in minimizing the functional
\[
x\mapsto \K{Tx-w_0}{Tx-w_0}_\KK + \rho\K{Vx-z_0}{Vx-z_0}_\EE.
\]
Defining the following indefinite inner product on $\KK\x \EE$:
\begin{equation}\label{met indef para KxE}
\K{(y,z)}{(y',z')}_{\rho}=\K{y}{y'}_\KK + \rho\K{z}{z'}_\EE,\quad\quad \textrm{$y,y'\in \KK$ and $z,z'\in\EE$},
\end{equation}
it is easy to see that $(\KK\x\EE,\K{\cdot}{\cdot}_\rho)$ is a Krein space. Also, if $L\in\mc{L}(\HH,\KK\x\EE)$ is defined by
\begin{equation}\label{L}
Lx=(Tx,Vx), \qquad x\in\HH,
\end{equation}
 the regularized problem can be restated as calculating
\begin{equation}\label{ILSP}
\min_{x\in\HH} \K{Lx-(w_0,z_0)}{Lx-(w_0,z_0)}_\rho,
\end{equation} 
which is the indefinite least-squares problem (without constraints) associated to the equation $Lx=(w_0,z_0)$. The existence of solutions to such problem was characterized in \cite{GMMP10_2,GMMP16} in terms of different properties of the range of $L$.

\medskip

We now present some preliminaries on Krein spaces.

\subsection{Preliminaries on Krein spaces}

In what follows we present the standard notation and some basic results on Krein spaces. For a complete exposition on the subject and the proofs of the results below see
\cite{Ando,Azizov,Bognar,Dritschel,Rovnyak}.

\medskip

An indefinite inner product space $(\mc{F}, \K{\cdot}{\cdot})$ is a (complex) vector space $\mc{F}$ endowed with a Hermitian sesquilinear form $\K{\cdot}{\cdot}: \mc{F}\x\mc{F} \ra \CC$.

A vector $x\in\mc{F}$ is {\it positive}, {\it negative}, or {\it neutral} if $\K{x}{x}>0$, $\K{x}{x}<0$, or $\K{x}{x}=0$, respectively.
The set of positive vectors in $\mc{F}$ is denoted by $\mathcal{P}^{++}(\mc{F})$, and the set of {\it nonnegative} vectors in $\mc{F}$
by $\mathcal{P}^+(\mc{F})$. The sets of negative, nonpositive and neutral vectors in $\mc{F}$ are defined analogously, and they are denoted by $\mathcal{P}^{--}(\mc{F})$, $\mathcal{P}^{-}(\mc{F})$, and $\mathcal{P}^0(\mc{F})$, respectively.

Likewise, a subspace $\M$ of $\mc{F}$ is {\it positive} if every $x\in\M$, $x\neq0$ is a positive vector in $\mc{F}$; and it is
{\it nonnegative} if $\K{x}{x} \geq0$ for every $x\in\M$. Negative, nonpositive and neutral subspaces are defined mutatis mutandis.

\medskip

%The following result is part of the folklore of the area, it can be traced back to mid of twentieth century, see \cite{F,KS,K}. The present statement is taken from Lemma 1.35 and Corollary 1.36 in \cite{Azizov}.
%
%\begin{lema}\label{lema Azizov}
%Suppose that $(\mc{F}_1,\K{\cdot}{\cdot}_1)$ is an indefinite inner product space, $(\mc{F}_2,\K{\cdot}{\cdot}_2)$ is an arbitrary inner product space and $T:\mc{F}_1\ra\mc{F}_2$ satisfies
%\[
%\K{Tx}{Tx}_2\geq 0 \qquad \text{for every $x\in \mathcal{P}^0(\mc{F}_1)$}.
%\]
%Then, for all $y\in \mathcal{P}^{--}(\mc{F}_1)$ and $z\in \mathcal{P}^{++}(\mc{F}_1)$,
%\[
%\frac{\K{Ty}{Ty}_2}{\K{y}{y}_1} \leq \frac{\K{Tz}{Tz}_2}{\K{z}{z}_1}.
%\]
%Under these conditions the finite limits 
%\[
%\mu_+(T)=\inf_{x\in \mathcal{P}^{++}(\mc{F}_1)} \frac{\K{Tx}{Tx}_2}{\K{x}{x}_1} \qquad \text{and} \qquad \mu_-(T)=\sup_{x\in \mathcal{P}^{--}(\mc{F}_1)} \frac{\K{Tx}{Tx}_2}{\K{x}{x}_1}
%\]
%exist. Moreover, $\mu_-(T)\leq \mu_+(T)$ and for any $\mu$ with $\mu_-(T)\leq \mu\leq \mu_+(T)$ the following inequality holds:
%\[
%\K{Tx}{Tx}_2\geq \mu \K{x}{x}_1 \qquad \text{for every $x\in\mc{F}_1$}.
%\]
%\end{lema}
%
%
%
%\medskip

If $\St$ is a subset of an indefinite inner product space $\mc{F}$, the \emph{orthogonal companion} to $\St$ is defined by 
\[
\St^{\ort}=\set{ x\in\mc{F} : \K{x}{s}=0 \; \text{for every $s\in\St$}}.
\]
It is easy to see that $\St^{\ort}$ is always a subspace of $\mc{F}$.

\begin{Def}
An indefinite inner product space $(\HH, \K{\cdot}{\cdot})$ is a \emph{Krein space} if it can be decomposed as a direct (orthogonal) sum of a Hilbert space and an anti Hilbert space, i.e. there
exist subspaces $\HH_\pm$ of $\HH$ such that $(\HH_+, \K{\cdot}{\cdot})$ and $(\HH_-, -\K{\cdot}{\cdot})$ are Hilbert spaces,
\begin{equation}\label{desc cano}
\HH=\HH_+ \dotplus \HH_-,
\end{equation}
and $\HH_+$ is orthogonal to $\HH_-$ with respect to the indefinite inner product. Sometimes we use the notation $\K{\cdot}{\cdot}_\HH$ instead of $\K{\cdot}{\cdot}$ to emphasize the Krein space considered.
\end{Def}

A pair of subspaces $\HH_\pm$ as in \eqref{desc cano} is called a \emph{fundamental decomposition} of $\HH$. Given a Krein space $\HH$ and a fundamental decomposition
$\HH=\HH_+\dotplus \HH_-$, the direct (orthogonal) sum of the Hilbert spaces $(\HH_+, \K{\cdot}{\cdot})$ and $(\HH_-, -\K{\cdot}{\cdot})$ is denoted
by $(\HH,\PI{\cdot}{\cdot})$.

If $\HH=\HH_+ \dotplus \HH_-$ and $\HH=\HH'_+ \dotplus \HH'_-$ are two different fundamental decompositions of $\HH$, the corresponding associated inner products $\PI{\cdot}{\cdot}$ and
$\PI{\cdot}{\cdot}'$ turn out to be equivalent on $\HH$. Therefore, the norm topology on $\HH$ does not depend on the chosen fundamental decomposition.

\medskip

If $(\HH,\K{\cdot}{\cdot}_\HH)$ and $(\KK,\K{\cdot}{\cdot}_\KK)$ are Krein spaces, $\mc{L}(\HH, \KK)$ stands for the vector space of linear transformations which are
bounded with respect to any of the associated Hilbert spaces $(\HH,\PI{\cdot}{\cdot}_\HH)$ and $(\KK,\PI{\cdot}{\cdot}_\KK)$. 
Given $T\in \mathcal{L}(\HH,\KK)$, the adjoint operator of $T$ (in the Krein spaces sense) is the unique operator $T^\#\in \mc{L}(\KK, \HH)$ such that
\[
\K{Tx}{y}_\KK=\K{x}{T^\#y}_\HH, \ \ \ \ x\in\HH ,y\in\KK.
\]
%An operator $T\in \mathcal{L}(\HH)$ is selfadjoint if $T=T^\#$. 

We frequently use that if $T\in \mc{L}(\HH,\KK)$ and $\M$ is a closed subspace of $\KK$ then 
\begin{equation}\label{preimag}
T^\#(\M)^{\ort_\HH}= T^{-1}(\M^{\ort_\KK}).
\end{equation}

Given a subspace $\M$ of a Krein space $\HH$, the {\it isotropic part} of $\M$ is defined by $\M^\circ:=\M\cap\M^{\ort}$. Then, $\M$ is {\it nondegenerate} if $\M^\circ=\set{0}$.

A subspace $\M$ of a Krein space $\HH$ is {\it pseudo-regular} if $\M+\M^{\ort}$ is a closed subspace of $\HH$, and it is {\it regular} if $\M+\M^{\ort}=\HH$. Regular subspaces are examples of
nondegenerate subspaces, but pseudo-regular subspaces can be degenerate ones. However, if $\M$ is a pseudo-regular subspace then there exists a regular subspace $\mc{R}$ such that 
\begin{equation}\label{desc}
\M=\M^\circ \sdo\ \mc{R},
\end{equation} 
where $\sdo$ stands for the direct orthogonal sum with respect to the indefinite inner product $\K{\cdot}{\cdot}$.
In fact, any closed subspace $\mc{R}$ such that $\M=\M^\circ \dotplus \mc{R}$ satisfies \eqref{desc} and
%and every closed subspace $\mc{R}$ satisfying \eqref{desc} 
turns out to be a regular subspace of $\HH$, see e.g. \cite{G84}. 
Note that a subspace $\M$ is regular if and only if it is pseudo-regular and nondegenerate.

The following propositions can be found in \cite[Lemma 3.4]{GMMP16} and \cite[Chapter 1, \S 7]{Azizov}, respectively.

\begin{prop}\label{preeliminares pseudo regular}
Given Krein spaces $\HH$ and $\KK$, let $T\in \mc{L}(\HH,\KK)$ with closed range. Then, $R(T)$ is pseudo-regular if and only if $R(T^\#T)$ is closed. 
\end{prop}

A subspace $\M$ of a Krein space $(\HH,\K{\cdot}{\cdot})$ is {\it uniformly positive} if %, fixing any inner product induced norm $\|\cdot\|$, 
there exists $\alpha>0$ such that 
\[
\K{x}{x}\geq\alpha\|x\|^2\quad \text{ for every $x\in\M$},
\]
where $\|\cdot\|$ is the norm of any associated Hilbert space. Uniformly negative subspaces are defined mutatis mutandis.

\begin{prop}\label{preeliminares regular}
Let $\M$ be a subspace of a Krein space $\HH$. Then, $\M$ is closed and uniformly positive (resp. negative) if and only if $\M$ is regular and nonnegative (resp. nonpositive).
\end{prop}

\subsection{The operator pencil associated to the QP1EQC}\label{props rango de L}

As we mentioned before, the regularization of Problem \ref{pb 1} is the indefinite least-squares problem (without constraints) determined by the indefinite inner product $\K{\cdot}{\cdot}_\rho$ in $\KK\x\EE$ and the operator $L$, see \eqref{met indef para KxE},\eqref{L}, and \eqref{ILSP}. In particular, if $R(L)$ is a (closed) nonnegative pseudo-regular subspace of $\KK\x\EE$ and $(w_0,z_0)\in R(L)+R(L)^{\ort}$ the vectors attaining \eqref{ILSP} are exactly the solutions of the {\it normal equation}: 
\[
L^\#(Lx-(w_0,z_0))=0,
\]
where $L^\#$ is the adjoint of $L$ with respect to the indefinite inner product $\K{\cdot}{\cdot}_\rho$, see \cite[Theorem 3.5]{GMMP16}.

\medskip

The adjoint operator of $L$ %with respect to the indefinite inner product $\K{\cdot}{\cdot}_\rho$ in $\KK\x\EE$ 
is given by
\[
L^\#(y,z)=T^\#y + \rho V^\#z, \quad (y,z)\in\KK\x\EE.
\]
Indeed, if $x\in\HH$ and $(y,z)\in\KK\x\EE$, we have that
\begin{align*}
\K{Lx}{(y,z)}_\rho &=  \K{Tx}{y}_\KK + \rho\K{Vx}{z}_\EE=\PI{x}{T^\#y} + \rho\PI{x}{V^\#z} \\ &=\PI{x}{T^\#y + \rho V^\#z}.
\end{align*}
Moreover, it is easy to see that 
\begin{equation*}
L^\#L=T^\#T+ \rho V^\#V,
\end{equation*}
and the normal equation turns into
\[
\big(T^\#T+ \rho V^\#V\big)x=T^\#w_0 +\rho V^\#z_0.
\]
Hence, the variation of the regularization parameter $\rho$ brings into consideration the study of the operator pencil $P(\la)= T^\#T+ \la V^\#V$ determined by the selfadjoint operators $T^\#T$ and $V^\#V$.

\bigskip

From the inclusions $N(L)\subseteq N(L^\#L)$ and $R(L^\#L)\subseteq R(L^\#)$ it is immediate that
$$N(T)\cap N(V)\subseteq N(L^\#L)\quad\text{and}\quad R(L^\#L)\subseteq N(T)^\bot+N(V)^\bot.$$ 

The next result describes in detail different properties of $R(L)$ as a subspace of the Krein space $(\KK\x\EE,\K{\cdot}{\cdot}_\rho)$, in terms of the operators $T$ and $V$. 

\begin{prop}\label{props rho fijo}
Given closed range operators $T\in\mc{L}(\HH,\KK)$ and $V\in \mc{L}(\HH,\EE)$, consider the regularization operator $L\in\mc{L}(\HH,\KK\x\EE)$. Then,
   \begin{enumerate}[label=\roman*)]
	\item $R(L)$ is closed if and only if $N(T)+N(V)$ is a closed subspace;
	%\item $R(L)$ is nonnegative if and only if $L^\#L$ is positive semidefinite;
	\item $R(L)$ is nondegenerate if and only if $N(L^\#L)=N(T)\cap N(V)$;
	\item $R(L)$ is pseudo-regular if and only if $N(T)+N(V)$ and $R(L^\#L)$ are closed subspaces.
%	\item $R(L)$ is regular if and only if $R(L^\#L)=R(T^\#T)+R(V^\#V)$.
\end{enumerate}
\end{prop}

\begin{proof} 
\noi {\it i.} \ On the one hand, $R(L)$ is closed if and only if $R(L^\#)$ is closed. On the other hand,
\[
R(L^\#)=R(T^\#) + R(V^\#) =N(T)^\bot + N(V)^\bot.
\]
Then, $R(L)$ is closed if and only if $N(T)^\bot + N(V)^\bot$ is closed. By \cite{Deutsch}, this is also equivalent to $N(T)+N(V)$ being a closed subspace.

\smallskip

%\noi {\it 2.} \ This is trivial.

%\smallskip

\noi {\it ii.} \ First, let us show that $R(L)^\circ= L(N(L^\#L))$. To this end, note that 
\[
N(L^\#L)=L^{-1}(N(L^\#))=L^{-1}\big(R(L)^{\ort}\big)=L^{-1}(R(L)^{\circ}).
\]
Applying $L$ to both sides of the above equality, we get that $L(N(L^\#L))=R(L)^\circ$. Therefore, $R(L)^\circ=\{0\}$ if and only if 
\[
N(L^\#L)\subseteq N(L)=N(T)\cap N(V).
\]
Since the inclusion $N(L)\subseteq N(L^\#L)$ always holds, the desired equivalence follows.

\smallskip

\noi {\it iii.} \ By item {\it i}, $L$ is a closed range operator if and only the sum of the nullspaces of $T$ and $V$ is closed in $\HH$.
But, for a closed range operator $L$, the pseudo-regularity of $R(L)$ is equivalent to the closedness of $R(L^\#L)$,
see Proposition \ref{preeliminares pseudo regular}.
%
%\smallskip
%
%\noi {\it 5.} \ Assume that $R(L)$ is regular. By items {\it 1.} and {\it 4.}, we have that 
%\[
%R(L^\#L)=N(L^\#L)^\bot=\ol{N(T)^\bot + N(V)^\bot}=N(T)^\bot + N(V)^\bot.
%\]
%Since $R(T^\#T)\subseteq \ol{R(T^\#)}=N(T)^\bot$ and $R(V^\#V)\subseteq \ol{R(V^\#)}=N(V)^\bot$, it turns out that
%\[
%R(T^\#T)+R(V^\#V)\subseteq N(T)^\bot + N(V)^\bot=R(L^\#L),
%\]
%and the other inclusion is immediate.
%
%Conversely, assume that $R(T^\#T)+R(V^\#V)=R(L^\#L)$. Observe that $R(L)$ is regular if and only if for each $(y,z)\in\KK\x\EE$ there exists $x_0\in\HH$ such that 
%\[
%Lx_0-(y,z)\in R(L)^{\ort}.
%\]
%Since $T$ and $V$ are closed range operators, the same happen to $T^\#$ and $V^\#$. Then, given $(y,z)\in \KK\x\EE$, there exist $y_0\in N(T^\#)^\bot=R(T)$, $y_1\in N(T^\#)$, $z_0\in N(V^\#)^\bot=R(V)$, $z_1\in N(V^\#)$ such that $y=y_0 + y_1$ and $z=z_0 + z_1$. Hence,
%\[
%L^\#(y,z)=T^\#y + \rho V^\# z=T^\#y_0 + \rho V^\# z_0\in R(T^\#T) + R(V^\#V)=R(L^\#L),
%\] 
%i.e. there exists $x_0\in \HH$ such that $L^\#(y,z)=L^\#Lx_0$. Therefore, $(y,z)-Lx_0\in N(L^\#)=R(L)^{\ort}$ and $R(L)$ is a regular subspace.
\end{proof}

As a corollary, in the particular case that $V$ is surjective, the regularity of the range of $L$  can be also
characterized by a range inclusion.

\begin{cor}\label{cor:regular}
Assume that $V$ is surjective and consider $\rho\neq0$. Then, $R(L)$ is a regular subspace of $(\KK\x\EE,\K{\cdot}{\cdot}_\rho)$ if and only if
\[
R(T^\#)\subseteq R(L^\#L).
\]
%the following conditions are equivalent: 
%
	%\item $R(L)$ is a regular subspace of $(\KK\x\EE,\K{\cdot}{\cdot}_\rho)$; 
	%\item $R(T^\#)\subseteq R(L^\#L)$.
%\end{enumerate}
%In this case, $N(L^\#L)=N(T)\cap N(V)$. 
%\[
%R(L^\#L)=R(T^\#T)+R(V^\#V)\quad\text{and}\quad N(L^\#L)=N(T)\cap N(V). 
%\]
\end{cor}

\begin{proof} 
If $R(L)$ is a regular subspace, then $R(L)$ is nondegenerate and pseudo-regular. Applying Proposition \ref{props rho fijo},
it follows that $N(L^\#L)=N(T)\cap N(V)$ and $R(L^\#L)$ is closed.
Hence,
\[
R(T^\#)=N(T)^\bot\subseteq \parentesis{N(T)\cap N(V)}^\bot=N(L^\#L)^\bot=R(L^\#L).
\]

%\marginpar{Acá estamos suponiendo que $V$ es suryectivo? No entiendo porqué $R(V^\#)=R(V^\#V)$.} 
Conversely, assume that $R(T^\#)\subseteq R(L^\#L)$. Then
$R(T^\#T)\subseteq R(T^\#)\subseteq R(L^\#L)$, or equivalently, $R(V^\#)= R(V^\#V)\subseteq R(L^\#L)$.
Hence, $R(L^\#L)=R(T^\#) + R(V^\#)=R(L^\#)$ and given $(y,z)\in\KK\x\EE$
there exists $x_0\in\HH$ such that $Lx_0-(y,z)\in N(L^\#)=R(L)^\ort$. Therefore, $(y,z)=Lx_0 + \big((y,z)-Lx_0\big)\in R(L)+R(L)^{\ort}$.  
%To prove that $R(L)$ is regular we have to show that for each $(y,z)\in\KK\x\EE$
%there exists $x_0\in\HH$ such that $Lx_0-(y,z)\in R(L)^\ort$, or equivalently,
%\[
%T^\#y+\rho V^\#z=L^\#(y,z)=L^\#Lx_0.
%\]
%Let $(y,z)\in\KK\x\EE$. Since both $R(T^\#)$ and $R(V^\#)$ are contained in $R(L^\#L)$, there exist $y_0,z_0\in\HH$ such that $T^\#y=L^\#Ly_0$ and
%$V^\#z=L^\#Lz_0$. Hence, $T^\#y+\rho V^\#z=L^\#L(y_0+\rho z_0)\in R(L^\#L)$ and the assertion is proved.
\end{proof}

Considering that $R(T^\#)=N(T)^\bot$ and $R(V^\#)=N(V)^\bot$, the proof of the above corollary implies the following.

\begin{cor}\label{cor:suma_nucleos}
Assume that $V$ is surjective. Given $\rho\neq0$, $R(L)$ is a regular subspace of $(\KK\x\EE,\K{\cdot}{\cdot}_\rho)$ if and only if 
\[
R(L^\#L)=N(T)^\bot+N(V)^\bot.
\]
%the following conditions are equivalent: 
%   \begin{enumerate}[label=\roman*.]
	%\item $R(L)$ is a regular subspace of $(\KK\x\EE,\K{\cdot}{\cdot}_\rho)$; 
	%\item $R(L^\#L)=N(T)^\bot+N(V)^\bot$.
%\end{enumerate}
\end{cor}

% \medskip

% \marginpar{Hay que pensar cómo poner esto}
% \begin{prop}\label{L surye}
% Assume that $T\in\mc{L}(\HH,\KK)$ and $V\in \mc{L}(\HH,\EE)$ are surjective operators. Then, 
% $R(L)=\KK\x\EE$ if and only if $\HH=N(T) + N(V)$.
% \end{prop}

% \begin{proof}
% Assume that $L$ is surjective, and consider $(w,0)\in\KK\x\EE$. Then, there exists $y\in\HH$ such that $(Ty,Vy)=Ly=(w,0)$. Hence, $y\in N(V)$ and, since $w\in\KK$ is arbitrary, it follows that $\KK=T(N(V))$. Therefore, $\HH=T^{-1}\big(T(N(V))\big)=N(T)+ N(V)$.

% Conversely, assume that $\HH=N(T) + N(V)$. Thus, $N(T)^\bot\cap N(V)^\bot=\set{0}$. If $(u,v)\in R(L)^{\ort}=N(L^\#)$ then
% \[
% T^\#u + \rho V^\#v=L^\#(u,v)=0.
% \]
% Hence, $T^\#u=-\rho V^\#v\in R(T^\#)\cap R(V^\#)=N(T)^\bot\cap N(V)^\bot=\set{0}$ and, by the injectivity of $T^\#$ and $V^\#$, we have that $(u,v)=(0,0)$. Therefore, $N(L^\#)=\set{(0,0)}$, or equivalently, $\ol{R(L)}=\KK\x\EE$. 

% But also $\HH=N(T) + N(V)$ implies the closedness of $R(L)$, see item {\it 1.} in Proposition \ref{props rho fijo}.
% \end{proof}

% \begin{obs}
% {\it If $T\in\mc{L}(\HH,\KK)$ and $V\in\mc{L}(\HH,\EE)$ are surjective operators, then $R(L)$ is a proper subspace of $\KK\x\EE$}.

% Indeed, by Proposition \ref{L surye}, $L$ is surjective if and only if $\HH= N(T) + N(V)$. But in this case $\KK=T(N(V))=T(\CV)$ is a nonnegative subspace of $\KK$, i.e. $\KK$ is a Hilbert space. This is a contradiction to the assumption that $T^\#T$ is indefinite.
% \end{obs}

\medskip

The above results give rise to a natural question: if the range of $L$ has some of the listed properties for a particular choice of $\rho$, does the property remain under perturbations of the parameter? The next section is devoted to analyzing this question.

\subsection{Varying the regularization parameter} 

Now we analyze the properties of the range of the regularization operator $L$ in the family of Krein spaces $(\KK\x\EE, \K{\cdot}{\cdot}_\rho)$ obtained under the variation of the regularization parameter $\rho$. 

\begin{prop}\label{todos cerrados}
If $R(L)$ is a closed subspace of $(\KK\x\EE, \K{\cdot}{\cdot}_{\rho_0})$ for some $\rho_0\neq 0$, then $R(L)$ is a closed subspace of $(\KK\x\EE, \K{\cdot}{\cdot}_{\rho})$ for each $\rho\neq 0$.
\end{prop}

\begin{proof}
This is an immediate consequence of item {\it i} in Proposition \ref{props rho fijo}, because the closedness of $N(T)+N(V)$ does nos depend on $\rho$.
\end{proof}

In the following we introduce an interval of {\it admissible values} for the regularization parameter, which is determined by Problem \ref{pb 1}. Let $\CV$ denote the set of neutral vectors of the quadratic form associated to $V^\#V$:
\[
\CV=Q(V^\#V)=\setB{y\in\HH\,\,:\,\,\K{Vy}{Vy}=0}.
\]
If $V^\#V$ is a positive (or negative) semidefinite operator in $\HH$, then $\CV$ coincides with $N(V)$.
But, if $V^\#V$ is \emph{indefinite}, the set $\CV$ is strictly larger than $N(V)$.
From now on $V^\#V$ is assumed to be indefinite; i.e. neither positive nor negative semidefinite.
A necessary condition for the existence of solutions for Problem \ref{pb 1} is that $T$ maps the set $\CV$ into the set of nonnegative vectors of the Krein space $\KK$, see \cite[Corollary 3.2]{GZMMP22}.   

In the following we rewrite several results from the previous section in the Krein space setting. The next result can be interpreted as
another manifestation of  Farkas' lemma, see \cite{Polik,Xia}. 
%It is an immediate consequence of Proposition \ref{Azizov pencils}. 
%For its proof, see Lemma 1.35 and Corollary 1.36 in \cite[Chapter 1, \S 1]{Azizov} or \cite[Lemma 6.1]{IKL82}.

\begin{prop}\label{prop:asimov}
$T(\CV)$ is a nonnegative set of $\KK$ if and only if there exists $\rho\in\RR$ such that $T^\#T+\rho V^\#V\in\mc{L}(\HH)^+$.
\end{prop}

% \begin{proof}
% Assume that $T(\CV)$ is a nonnegative set of $\KK$. Then, 
% \[
% \PI{T^\#T\, x}{x}=\K{Tx}{Tx}_\KK\geq 0 \qquad \text{ for every $x\in\CV=Q(V^\#V)$}.
% \]
% Hence, by Proposition \ref{Azizov pencils}, $I_\geq(T^\#T,V^\#V)\neq\varnothing$. The converse is trivial.
%If $\HH$ is endowed with the indefinite inner product given by
%\[
%\K{x}{y}_\HH:= \K{Vx}{Vy}_\EE, \qquad x,y\in\HH,
%\]
%then, by Lemma \ref{lema Azizov}, there exists $\mu\in\RR$ such that 
%\[
%\K{Tx}{Tx} \geq \mu \K{x}{x}_\HH \quad \text{for every $x\in\HH$}.
%\]
%Considering $\rho:=-\mu$, the above inequality is equivalent to say that $T^\#T+\rho V^\#V\in\mc{L}(\HH)^+$. The converse is trivial.
% \end{proof}

If it is also assumed that $T^\#T$ is indefinite, it is easy to see that $T(\CV)$ is a nonnegative set of $\KK$ if and only if $R(L)$ is nonnegative in $(\KK\x\EE,\K{\cdot}{\cdot}_\rho)$ for some $\rho\neq 0$.

Let us also consider the subsets of $\HH$ where the quadratic form associated to $V^\#V$ takes positive and negative values:
\begin{align}\label{eq:def_PV}
\PV:= &\set{x\in\HH\,:\,\K{Vx}{Vx}>0}\quad\text{and}\\ \PVV:= &\set{x\in\HH\,:\,\K{Vx}{Vx}<0}. \nonumber
\end{align}
The next corollary determines the interval for the parameter $\rho$ in which $R(L)$ is a nonnegative subspace of $(\KK\x\EE,\K{\cdot}{\cdot}_\rho)$.

\begin{cor}\label{cor:intervalo}
Assume that $T(\CV)$ is a nonnegative set of $\KK$, and define
\begin{equation}\label{rhos}
\rho_-:=-\inf_{x\in\PV}\frac{\K{Tx}{Tx}}{\K{Vx}{Vx}}\quad\text{and}\quad\rho_+:=-\sup_{x\in\PVV}\frac{\K{Tx}{Tx}}{\K{Vx}{Vx}}.
\end{equation}
Then $\rho_-<+\infty$, $\rho_+>-\infty$ and $\rho_-\leq\rho_+$.
Moreover, $T^\#T+\rho V^\#V\in\mc{L}(\HH)^+$ if and only if $\rho\in[\rho_-,\rho_+]$.
\end{cor}

\medskip

Hereafter we make the following assumptions.

\begin{hyp}\label{hyp:iniciales}
$T^\#T$ and $V^\#V$ are indefinite operators acting in $\HH$ and 
\[
T(\CV)\ \ \text{is a nonnegative subset of $\KK$}.
\]
\end{hyp}

The indefiniteness of  $T^\#T$ implies that $0\notin[\rho_-,\rho_+]$, i.e. the interval $[\rho_-,\rho_+]$ is either
contained in the positive or in the negative real axis.
In fact, it can be contained in either one of them,
as the following example illustrates.

\begin{example}\label{eje} 
Consider surjective operators $T\in\mc{L}(\HH,\KK)$ and $V\in\mc{L}(\HH,\EE)$ such that
$T^\#T$ and $V^\#V$ can be represented, according to some decomposition $\HH=\HH_+\oplus\HH_-$, as
\[
T^\#T=\begin{bmatrix}
\alpha I&0\\
0&\beta I\end{bmatrix}\quad\text{and}\quad V^\#V=\begin{bmatrix}I&0\\0&-I\end{bmatrix},
\]
with $\alpha,\beta\in\RR$. If $\rho\in\RR$, and $x\in\HH$ is written as $x=x_++x_-$ with $x_\pm\in\HH_\pm$,
then
\[
\PI{(T^\#T+\rho V^\#V)x}{x}=(\alpha+\rho)\|x_+\|^2+(\beta-\rho)\|x_-\|^2.	
\]
Thus, $T^\#T+\rho V^\#V\in\mc{L}(\HH)^+$ if and only if $\alpha+\rho\geq0$ and $\beta-\rho\geq0$. Suppose that $\alpha=-1$ and $\beta=2$, then it is immediate
that $\rho_-=1$ and $\rho_+=2$. If we now consider the case when $\alpha=2$ and $\beta=-1$, then $\rho_-=-2$ and $\rho_+=-1$.
\end{example}

\begin{obs}%\label{obs:pencils}
If we denote $\CT=Q(T^\#T)=\{x\in\HH: \K{Tx}{Tx}_\KK=0\}$, Hypotheses \ref{hyp:iniciales} also imply that $V(\CT)$ is either a nonnegative or a nonpositive set of $\EE$,
depending on whether $[\rho_-,\rho_+]$ is contained in the positive or in the negative real axis. Indeed, if
$\rho\in[\rho_-,\rho_+]$ then
\[
\rho\K{Vx}{Vx}=\PI{(T^\#T+\rho V^\#V)x}{x}\geq 0,\quad\quad\text{for all $x\in\CT$}.
\]
\end{obs}

%\begin{Def}
%The interval $[\rho_-,\rho_+]$ determined by \eqref{rhos} is the {\it interval of admissible values} for the regularization parameter.
%\end{Def}

Now, we apply the results obtained in Section \ref{Pencils} to the operator pencil determined by the selfadjoint operators $A:=T^\#T$ and $B:=V^\#V$. We start by analyzing the behaviour of the nullspace of $L^\#L=T^\#T+\rho V^\#V$ for each $\rho\in[\rho_-,\rho_+]$. 
As a consequence of Proposition \ref{nucleo del interior}, if $\rho_-\neq\rho_+$ then the nullspace of $L^\#L$ does not depend on $\rho$ when $\rho\in(\rho_-,\rho_+)$.

\begin{lema}\label{lema:nucleo_abierto} 
If $\rho_-\neq\rho_+$ then 
\[
N(T^\#T+\rho V^\#V)=\CT\cap\CV=N(T^\#T)\cap N(V) \quad\text{for all}\quad\rho\in(\rho_-,\rho_+).
\]
\end{lema}

%\begin{proof}
%Let $\rho\in(\rho_-,\rho_+)$. In this case Corollary \ref{cor:intervalo} asserts that
%$T^\#T+\rho V^\#V\in\mc{L}(\HH)^+$, which implies that 
%\[
%\CT\cap\CV=N(T^\#T+\rho V^\#V)\cap\CV. 
%\]
%Now assume that there exists $x_0\in N(T^\#T+\rho V^\#V)\setminus(\CT\cap\CV)=
%N(T^\#T+\rho V^\#V)\setminus\CV$. Since $\K{Tx_0}{Tx_0}+\rho\K{Vx_0}{Vx_0}=0$ and $\rho_-<\rho<\rho_+$,
%\begin{equation}\label{eq_desigualdades}
%\sup_{x\in\PVV}\frac{\K{Tx}{Tx}}{\K{Vx}{Vx}}=-\rho_+<\frac{\K{Tx_0}{Tx_0}}{\K{Vx_0}{Vx_0}}<-\rho_-=\inf_{x\in\PV}\frac{\K{Tx}{Tx}}{\K{Vx}{Vx}}.
%\end{equation}
%Thus, since $x_0\in\PV$ or $x_0\in\PVV$ \eqref{eq_desigualdades} leads to a contradiction, and consequently 
%\[
%N(T^\#T+\rho V^\#V)=\CT\cap\CV.
%\]
%Then, $N(T^\#T+\rho V^\#V)$ does not depend on $\rho$, and if $x\in\CT\cap\CV$ then $T^\#Tx=-\rho' V^\#Vx$ for any $\rho'\in(\rho_-,\rho_+)$.
%This implies that if $x\in N(T^\#T+\rho V^\#V)$ then $x\in N(T^\#T)\cap N(V^\#V)=N(T^\#T)\cap N(V)$.
%\end{proof}

%There is another immediate consequence of Lemma \ref{lema:nucleo_abierto}.

%\begin{cor}\label{cor_rango} If $\rho_-\neq \rho_+$ then
%\[
%\overline{R(T^\#T+\rho V^\#V)}=\overline{N(T^\#T)^\bot+N(V)^\bot}\quad\quad\text{for all $\rho\in(\rho_-,\rho_+)$}.
%\]
%\end{cor}

Assume that $\rho_-\neq\rho_+$. To characterize the nullspace of the operators associated to the extreme values of the interval, note that the sets defined in \eqref{Ms} are given by:
%i.e. $N(T^\#T+\rho_- V^\#V)$ and $N(T^\#T+\rho_+ V^\#V)$, we introduce the following sets:
\begin{align*}
\mc{M}_+= \displaystyle{\set{x\in\PV\,:\,\frac{\K{Tx}{Tx}}{\K{Vx}{Vx}}=-\rho_-}}
\end{align*}
and
\begin{align*}
\mc{M}_-= &\displaystyle{\set{x\in\PVV\,:\,\frac{\K{Tx}{Tx}}{\K{Vx}{Vx}}=-\rho_+}}.
\end{align*}

%Although $\M_\pm$ is not a subspace, if $\M_\pm\neq\varnothing$ then $\M_\pm + N(T^\#T)\cap N(V)=\M_\pm$.

\begin{lema}\label{lema:nucleo_rho+}
The following conditions are satisfied:
   \begin{enumerate}[label=\roman*)]
\item if $\rho_-\neq\rho_+$, then $N(T^\#T+\rho_\pm V^\#V)=\mc{M}_\mp\cup \parentesis{N(T^\#T)\cap N(V)}$;
\item if $\rho_-=\rho_+$, then $\M_+=\M_-$ and $N(T^\#T+\rho_\pm V^\#V)=\mc{M}_+\cup \parentesis{\CT\cap\CV}$.
\end{enumerate}
\end{lema}

Combining Lemma \ref{lema:nucleo_abierto} with item {\it ii} in Proposition \ref{props rho fijo}, we have:

\begin{prop}\label{prop_nucleo} 
If $\rho_-\neq\rho_+$ then the following conditions are equivalent:
   \begin{enumerate}[label=\roman*)]
\item $R(L)$ is a nondegenerate subspace of $(\KK\x\EE,\K{\cdot}{\cdot}_\rho)$ for some $\rho\in(\rho_-,\rho_+)$;
\item $R(L)$ is a nondegenerate subspace of $(\KK\x\EE,\K{\cdot}{\cdot}_\rho)$ for all $\rho\in(\rho_-,\rho_+)$;
\item $N(T^\#T)\cap N(V)=N(T)\cap N(V)$.
\end{enumerate}
\end{prop}

\subsection{Pseudo-regularity and regularity}

In order to prove the main result of this section, we first show that pseudo-regularity is an invariant property in $(\rho_-,\rho_+)$.

% we recall the relationship between the seminorms associated to the family of positive operators
% $T^\#T+\rho V^\#V$, for $\rho\in[\rho_-,\rho_+]$.
% Given a fixed $\rho\in[\rho_-,\rho_+]$, consider the seminorm $\|\cdot\|_\rho$ defined by
% \begin{equation}\label{norma rho}
% \|x\|_\rho:=\PI{(T^\#T+\rho V^\#V)x}{x}^{1/2},\quad x\in\HH,
% \end{equation}
%yields $(\HH,\|\cdot\|_\rho)$ is a seminormed space. 
% the closed subspace
% \begin{equation}\label{H'}
% \HH'=\ol{N(T^\#T)^\bot+N(V^\#V)^\bot}, 
% \end{equation}
% and assume that $\rho_-\neq\rho_+$. If $\rho\in(\rho_-,\rho_+)$  we have that $N(T^\#T+\rho V^\#V)=N(T^\#T)\cap N(V^\#V)$, see Lemma \ref{lema:nucleo_abierto}. Thus, %$\HH'$ is a dense subspace of $N(T^\#T+\rho V^\#V)^\bot$ (with respect to the norm $\|\cdot\|$). 
% \begin{equation}\label{H' es la clausura}
% \HH'= N(T^\#T+\rho V^\#V)^\bot \qquad \text{for every $\rho\in(\rho_-,\rho_+)$.}
% \end{equation}
 
% By Lemma \ref{lema:normas_equivalentes_p}, $(\HH',\|\cdot\|_\rho)$ is a normed space for all $\rho\in(\rho_-,\rho_+)$ 
% and the norms associated to the different values of the regularization parameter (with the exception of the extrema) are all equivalent on $\HH'$.

% \begin{lema}\label{prop:normas_equivalentes} 
% If $\rho_-\neq\rho_+$ then $\|\cdot\|_{\rho}$ and $\|\cdot\|_{\rho'}$ are equivalent on $\HH'$
% for every $\rho,\rho'\in(\rho_-,\rho_+)$.
% \end{lema}

%Now we are in position to prove the main result of this subsection.
% As a consequence of the above lemma, pseudo-regularity is an invariant property in $(\rho_-,\rho_+)$.

\begin{prop}\label{prop_todos_pseudo_regulares}
Assume that $\rho_-\neq\rho_+$. If $R(L)$ is a pseudo-regular subspace of $(\KK\x\EE,\K{\cdot}{\cdot}_{\rho_0})$ for some
$\rho_0\in(\rho_-,\rho_+)$, then it is a 
pseudo-regular subspace of $(\KK\x\EE,\K{\cdot}{\cdot}_\rho)$ for all $\rho\in(\rho_-,\rho_+)$.
\end{prop}

%\smallskip

\begin{proof} 
Assume that $R(L)$ is a pseudo-regular subspace of $(\KK\x\EE,\K{\cdot}{\cdot}_{\rho_0})$ for some $\rho_0\in(\rho_-,\rho_+)$.
By item {\it iii} in Proposition \ref{props rho fijo}, we have that $N(T)+N(V)$ and $R(T^\#T+\rho_0 V^\#V)$ are
closed subspaces of $\HH$.
%Hence, 
%since $T^\#T+\rho_0 V^\#V$ is a bounded and invertible operator on $\HH_{\rho_0}$,
%the norms $\|\cdot\|$ and $\|\cdot\|_{\rho_0}$ are equivalent on $\HH'$.
% $(T^\#T+\rho_0 V^\#V)^\dag$ is a bounded operator. Thus, for
% every $x\in \HH_{\rho_0}$, %considering that $\|x\|_{\rho_0}=\|(T^\#T+\rho_0 V^\#V)^{1/2}x\|$ yields
% \[
% \|(T^\#T+\rho_0 V^\#V)^\dag\|^{-1/2}\|x\|\leq\|x\|_{\rho_0}\leq\|T^\#T+\rho_0 V^\#V\|^{1/2}\|x\|,
% \]
% i.e. the norms $\|\cdot\|$ and $\|\cdot\|_{\rho_0}$ are equivalent on $\HH_{\rho_0}$. 
% Now, let $\rho\in(\rho_-,\rho_+)$. By Lemma \ref{lema:nucleo_abierto}, we have that
% $\ol{R(T^\#T+\rho V^\#V)}=R(T^\#T+\rho_0 V^\#V)=\HH'$. Also, Proposition \ref{prop:normas_equivalentes} says that $\|\cdot\|_\rho$ and $\|\cdot\|_{\rho_0}$ are equivalent on $\HH'$. Then, there exists $M>0$ such that
% \[
% M\|x\|^2\leq\|x\|_\rho^2=\|(T^\#T+\rho V^\#V)^{1/2}x\|^2 \quad \text{ for every $x\in\HH'$}.
% \]
If $\rho\in(\rho_-,\rho_+)$, applying Corollary \ref{rangos cerrados} $R(T^\#T+\rho V^\#V)$ is also a closed subspace of $\HH$. 
By item {\it iii} in Proposition \ref{props rho fijo}, it follows that $R(L)$ is also a pseudo-regular subspace of $(\KK\x\EE,\K{\cdot}{\cdot}_\rho)$.
\end{proof}

\begin{cor}\label{prop:todos_regulares} 
Suppose that $\rho_-\neq\rho_+$. If $R(L)$ is a regular subspace of $(\KK\x\EE,\K{\cdot}{\cdot}_{\rho_0})$ for some $\rho_0\in(\rho_-,\rho_+)$, then it is a
regular subspace of $(\KK\x\EE,\K{\cdot}{\cdot}_\rho)$ for all $\rho \in(\rho_-,\rho_+)$.
\end{cor}

%(closed) uniformly positive

\begin{proof}
Since a subspace is regular if and only if it is nondegenerate and pseudo-regular,
the statement is an immediate consequence of Proposition \ref{prop_nucleo} and Proposition \ref{prop_todos_pseudo_regulares}.
\end{proof}

%\bigskip
%\subsection{A deeper insight into the regular case}

%In these paragraphs 
Now, we deepen into the analysis of the case in which $R(L)$ is a regular subspace of
$(\KK\x\EE,\K{\cdot}{\cdot}_\rho)$ for some $\rho\in[\rho_-,\rho_+]$. First, we show that this regularity is equivalent to
$\CV$ being a uniformly positive subset of $\HH$ with respect to the form induced by $T^\#T$.
For the sake of simplicity, from now on we assume that
%along this section we assume that 
\[
N(T)\cap N(V)=\set{0}, 
\]
which is not necessarily true in the general case but does not imply a loss of generality. With minor adjustments the following results can be expressed for the general case.

\begin{prop}\label{prop:uniformemente} 
Assume that $N(T)\cap N(V)=\set{0}$. Then, the following conditions are equivalent:
   \begin{enumerate}[label=\roman*)]
\item there exists $\alpha>0$ such that $\K{Ty}{Ty}\geq\alpha\|y\|^2$ for every $y\in\CV$;
\item $T(\CV)$ is a uniformly positive set of $\KK$ and $N(T)+N(V)$ is a closed subspace;
\item there exists $\rho\in\RR$ such that $T^\#T+\rho V^\#V$ is a positive definite operator;
\item there exists $\rho\in\RR$ such that $R(L)$ is a (closed) uniformly positive subspace of $(\KK\x\EE,\K{\cdot}{\cdot}_\rho)$.
\end{enumerate}
\end{prop}

\begin{proof}
The equivalence {\it i$\leftrightarrow$iii} follows from Theorem \ref{pencil positivo} considering $A=T^\#T$ and $B=V^\#V$. %To complete the proof, we show that 
\smallskip

\noi {\it i$\to$ii.}\ \  Suppose that there exists $\alpha>0$ such that $\K{Ty}{Ty}\geq\alpha\|y\|^2$ for every $y\in\CV$. In particular, this implies that $T(N(V))$ is closed. Then, by Proposition \ref{R(AB)}, $N(T)+N(V)$ is also closed.

Also, $T(\CV)$ is a uniformly positive set of $\KK$ because
\[
\K{Ty}{Ty}\geq\alpha\|y\|^2\geq\frac{\alpha}{\|T\|^{2}}\|Ty\|^2,\quad\quad\text{for every $y\in\CV$.}
\]

\noi {\it ii$\to$iv.}\ \  Assume that there exists $\alpha>0$ such that
$\K{Ty}{Ty}\geq\alpha\|Ty\|^2$ for every $y\in\CV$. Then, if $A:=T^\#T - \alpha T^*T$ and $B:=V^\#V$ the inequality is equivalent to 
\[
\PI{Ay}{y} \geq 0 \qquad \text{for every $y\in\HH$ such that $\PI{By}{y}=0$}.
\]
Hence, by Proposition \ref{Azizov pencils},  there exists $\rho\in\RR$ such that $T^\#T-\alpha T^*T+\rho V^\#V= A+\rho B\geq 0$. So, $\alpha T^*T\leq L^\#L$. Then, by Theorem \ref{Doug} we have that
\begin{equation}\label{Emi}
N(T)^\bot=R(T^*)\subseteq R\big((L^\#L)^{1/2}\big).
\end{equation}
Now, we want to prove that $L^\#L$ is invertible, which implies that $R(L)$ is regular. To do so, we make use of the following result borrowed from \cite{CMS}: if $C\in \mc{L}(\HH)^+$ and $\St$ is a closed subspace of $\HH$ such that $\St\cap N(C)=\{0\}$ and $C(\St)$ is closed, then $\HH=\St^\bot + C(\St)$, see Proposition 3.3 and Lemma 3.8 in \cite{CMS}. 

For our purpose, consider the positive semidefinite operator $L^\#L$ and the closed subspace $N(T)$. On the one hand, the closedness of $N(T)+N(V)$ implies that $R(L)=N(T)^\bot+N(V)^\bot$ is also a closed subspace of $\HH$. Moreover, by Proposition \ref{R(AB)}, $L^\#L(N(T))=V^\#V(N(T))$ is a closed subspace of $\KK$. Hence, combining \eqref{Emi} with the above mentioned result, 
%Since $N(T)+N(V)$ is closed, $R(L)$ and  $L^\#L(N(T))=V^\#V(N(T))$ are closed subspaces of $\KK$ and $\HH$, respectively. Considering that $L^\#L$ is positive semidefinite
%then, by \cite[Lemma 3.8]{CMS} and \cite[Proposition 3.3]{CMS},
\[
\HH=N(T)^\bot\dotplus L^\#L(N(T))\subseteq R\big((L^\#L)^{1/2}\big).
\]
Also, it is immediate that $N(L^\#L)=N(T)\cap N(V)=\set{0}$.
Hence, $(L^\#L)^{1/2}$ is invertible.  
In particular, $R(L^\#L)=\HH$ and
\[
\KK\x\EE=(L^\#)^{-1}\big(R(L^\#L)\big)=R(L) + N(L^\#)=R(L) + R(L)^{\ort},
\]
i.e. $R(L)$ is a regular subspace of $\KK\x\EE$.

\smallskip

%This implies that $R(L^\#L)=\HH$ is a closed subspace. 
%Since $R(L)$ is closed because $N(T)+N(V)$ is closed, $R(L)$ is a regular subspace of $(\KK\x\EE,\K{\cdot}{\cdot}_\rho)$.

%Consider the inner products $\K{\cdot}{\cdot}_1$ and $\K{\cdot}{\cdot}_2$ on $\HH$
%given by
%\[
%\K{x}{y}_1=\K{Vx}{Vy}_\EE\,\,\text{and}\,\,
%\K{x}{y}_2=\K{Tx}{Ty}_\KK-\alpha\PI{Tx}{Ty}_\KK,\quad\quad x,y\in\HH,
%\]
%respectively. Since $\mc{P}_1^0=\CV$ is a nonnegative set in $\big(\HH,\K{\cdot}{\cdot}_2\big)$,
%Lemma \ref{lema Azizov} yields there exists $\rho\in\RR$ such that
%\[
%\PI{(T^\#T-\alpha T^*T+\rho V^\#V)x}{x}=\K{x}{x}_2+\rho\K{x}{x}_1\geq0,\quad\quad\text{for every $x\in\HH$}.
%\]
%Hence, $T^\#T-\alpha T^*T+\rho V^\#V$ is a positive semidefinite operator, and thus
%$\alpha T^*T\leq L^\#L$. By \cite[Theorem 1]{Douglas} this implies that
%$N(T)^\bot=R(T^*)\subseteq R\big((L^\#L)^{1/2}\big)$.
%Also, it is immediate that $N(L^\#L)=N(T)\cap N(V)=\set{0}$.
%Since $N(T)+N(V)$ is closed, $L^\#L(N(T))=V^\#V(N(T))$ is closed. Considering that $L^\#L$ is positive semidefinite
%then, by \cite[Lemma 3.8]{CMS} and \cite[Proposition 3.3]{CMS},
%\[
%\HH=N(T)^\bot\dotplus L^\#L(N(T))\subseteq R\big((L^\#L)^{1/2}\big).
%\]
%Hence $(L^\#L)^{1/2}$ is invertible, and then $L^\#L$ is also invertible.
%This implies that $R(L^\#L)=\HH$ is a closed subspace. Since $R(L)$ is closed because $N(T)+N(V)$ is closed, $R(L)$ is a regular subspace
%of $(\KK\x\EE,\K{\cdot}{\cdot}_\rho)$.

\noi {\it iv$\to$iii.}\ \  Suppose that there exists $\rho\in\RR$ such that $R(L)$ is a (closed) uniformly positive subspace of $(\KK\x\EE,\K{\cdot}{\cdot}_\rho)$.
Then, $R(L^\#L)=\HH$ and Proposition \ref{props rho fijo} ensures that $N(L^\#L)=N(T)\cap N(V)=\set{0}$.
Hence, $L^\#L$ is a positive definite operator.
%\smallskip
%
%\noi {\it 3.$\to$1.}\ \ Assume that there exists $\rho\in\RR$ such that
%$T^\#T+\rho V^\#V$ is a positive definite operator. Then
%the norms $\|\cdot\|$ and $\|\cdot\|_\rho$ are equivalent. Considering that
%$\|y\|_\rho^2=\K{Ty}{Ty}$ for every $y\in\CV$, it follows that there exists $\alpha>0$ such that
%\[
%\alpha\|y\|^2\leq\|y\|_\rho^2=\K{Ty}{Ty},\quad\quad\text{for every $y\in\CV$.}
%\]
\end{proof}

As a consequence of the proof, $\rho\in\RR$ satisfies item {\it iii}\,  if and only if it satisfies item {\it iv}. Then, combining the above proposition with Theorem \ref{pencil positivo}, we have the following result.

\begin{thm}\label{regular en el interior}
If $\rho\in[\rho_-,\rho_+]$ is such that $R(L)$ is a regular subspace of $(\KK\x\EE,\K{\cdot}{\cdot}_\rho)$, then $\rho_-\neq\rho_+$ and $\rho\in(\rho_-,\rho_+)$.
\end{thm}

Therefore, the regular case cannot occur in the boundary of the interval. The following corollary relates the lower bound that determines the uniformly positiveness of $T(\CV)$ (as a subset of $\KK$) with a closed interval contained in $(\rho_-,\rho_+)$ where the pencil is uniformly bounded from below.

%\bigskip
%
%Our next aim is to show that if $R(L)$ is a regular subspace of $(\KK\x\EE,\K{\cdot}{\cdot}_\rho)$ for some $\rho\in[\rho_-,\rho_+]$ then $\rho_-\neq\rho_+$ and $\rho\in(\rho_-,\rho_+)$, i.e. the regular case cannot occur in the boundary of the interval.

%As a consequence of this result we have the following corollary, which is the analogous version of Corollary \ref{cor:intervalo}.

\begin{cor}\label{cor:uniformemente} 
Given $\alpha>0$, the following conditions are equivalent:
   \begin{enumerate}[label=\roman*)]
\item $\K{Ty}{Ty}\geq\alpha\|y\|^2$ for every $y\in\CV$;
\item there exists an interval $[\eta_-,\eta_+]$
contained in $(\rho_-,\rho_+)$ such that $T^\#T+\rho V^\#V\geq\alpha I$ for every $\rho\in[\eta_-,\eta_+]$.
\end{enumerate}
\end{cor}

\begin{proof}
Firstly, {\it i}$\ra${\it ii} follows from the proof of Theorem \ref{pencil positivo}, considering $A:=T^\#T$ and $B:=V^\#V$. Secondly, if $\rho\in[\eta_-,\eta_+]$ it is immediate that, for every $y\in\CV$,
\[
\K{Ty}{Ty}=\PI{T^\#Ty}{y}=\PI{(T^\#T+\rho V^\#V)y}{y}\geq \alpha \PI{y}{y}=\alpha \|y\|^2. \qedhere
\]
\end{proof}

\medskip

\subsection*{Acknowledgements}

The authors gratefully acknowledge the support of CONICET through the
grant PIP 11220200102127CO. F.~Mart\'{\i}nez Per\'{\i}a also acknowledges the support from UNLP 11X829 and PICT 2015-1505.

\end{document}